\documentclass[11pt,reqno]{amsart}
\usepackage{etoolbox}
\makeatletter
\patchcmd{\@sect}
  {\@hangfrom{\hskip #3\relax\@svsec}{\interlinepenalty\@M #8\par}}
  {\ifnum#2=\@ne
     \centering\interlinepenalty\@M\hskip #3\relax\@svsec#8\par
   \else
     \@hangfrom{\hskip #3\relax\@svsec}{\interlinepenalty\@M #8\par}
   \fi}
  {}{\PackageError{manuscript}{Section heading patch failed}{}}
\makeatother
\usepackage[a4paper,textwidth=154mm,hcentering,top=28mm,bottom=28mm]{geometry}
\usepackage{amsmath,amssymb,amsthm,booktabs,hyperref}
\hypersetup{colorlinks=true,urlcolor=blue,linkcolor=blue,citecolor=blue,filecolor=blue,pdfborder={0 0 0}}
\numberwithin{equation}{section}
\newtheorem{theorem}{Theorem}[section]
\newtheorem{lemma}[theorem]{Lemma}
\newtheorem{proposition}[theorem]{Proposition}
\newtheorem{corollary}[theorem]{Corollary}
\theoremstyle{definition}

\theoremstyle{remark}

\theoremstyle{definition}
\newtheorem{question}[theorem]{Question}
\DeclareMathOperator{\tr}{tr}
\DeclareMathOperator{\Var}{Var}

\DeclareMathOperator{\TV}{TV}
\DeclareMathOperator{\KL}{KL}
\DeclareMathOperator{\cum}{cum}
\title[The spectral conjecture and three detection scales]{Counterexamples to the Mao--Wu--Xu spectral conjecture and a three-scale boundary for spherical random graphs}
\author{Congyi Luo}
\address{School of Data Science, Fudan University, Shanghai 200433, China}
\email{cyluo24@m.fudan.edu.cn}
\date{}
\begin{document}
\begin{abstract}
Random geometric graph detection asks whether a graph formed by sampling $n$ independent uniform points on $S^{d-1}$ and connecting them randomly according to their inner products can be distinguished, from its adjacency matrix alone, from an Erd\H{o}s--R\'enyi graph with the same edge density.
Mao, Wu, and Xu conjectured that the cubic trace of the standardized connection kernel determines the detection boundary: the limits zero and infinity of $n^3(\tr K^3)^2$ should correspond to impossibility and strong detection, respectively.
We construct counterexamples and prove a uniform result for spherical polynomial kernels of uniformly bounded degree whose mean edge density stays away from zero and one.
If $K$ is the integral operator of the standardized centered kernel, let
\[
 n_* =\min\left\{
 \frac{1}{|\tr K^3|^{2/3}},\quad
 \frac{1}{\sqrt{\tr K^4}},\quad
 \frac{d}{\tr K^2}
 \right\},
\]
where a term with zero denominator is $+\infty$.
The total variation distance between the geometric and independent-edge graphs tends to zero when $n/n_*\to0$, and to one when $n/n_*\to\infty$.
The three scales correspond to triangles, four-cycles, and global geometry.
In particular, the pure degree-four spherical harmonic kernel has only positive nonzero eigenvalues, yet its detection scale is $n\asymp d^5$.
Throughout $d^5\ll n\ll d^{16/3}$, the triangle and four-cycle signals and every fixed-degree standardized polynomial mean gap vanish, while the full graph remains strongly detectable.
Thus the failure of the cubic-trace conjecture is not solely due to spectral sign cancellation.
The impossibility proof combines higher-order spherical integration by parts, a graph expansion of cumulants, and a finite-order relative entropy comparison.
\end{abstract}
\maketitle
\section{Introduction}\label{sec:intro}
Mao, Wu, and Xu \cite{MWX} proposed a spectral conjecture for random geometric graphs:
the cubic trace of the standardized connection kernel should determine whether the geometric graph can be distinguished from an Erd\H{o}s--R\'enyi graph with the same edge density.
The conjecture reduces a statistical question concerning the entire adjacency matrix to a single spectral quantity.
We show that the cubic trace alone does not describe general sequences of spherical kernels,
and establish a detection boundary consisting of three explicit scales for dense polynomial kernels of uniformly bounded degree.

\subsection{Detection and the Mao--Wu--Xu spectral conjecture}
Let $X_1,\ldots,X_n$ be independent and uniform on the unit sphere $S^{d-1}\subset\mathbb R^d$, where $d\ge2$.
Conditional on these latent points, generate undirected edges independently according to
\[
 A_{ij}\mid X_1,\ldots,X_n
 \sim\operatorname{Bernoulli}\bigl(W(\langle X_i,X_j\rangle)\bigr),\qquad i<j,
\]
where $W:[-1,1]\to[0,1]$ is the connection function.
Write $P_{n,d,W}$ for the graph distribution and set
$p=\mathbb EW(\langle X_1,X_2\rangle)$.
Only the adjacency matrix is observed. We compare $P_{n,d,W}$ with $G(n,p)$;
the parameters of both models are known, whereas the latent points are unobserved.

Throughout, assume $0<p<1$. If $p=0$ or $p=1$, both models are respectively the empty or complete graph and hence have the same distribution.
Let $\mu_d$ be the uniform probability measure on the sphere. Define the standardized kernel and its integral operator by
\begin{equation}\label{eq:normalized-operator}
 \kappa(t)=\frac{W(t)-p}{\sqrt{p(1-p)}},\qquad
 (Kf)(x)=\int\kappa(\langle x,y\rangle)f(y)\,d\mu_d(y).
\end{equation}
The operator $K$ is self-adjoint and annihilates constant functions.
The notation $\tr K^r$ always denotes the operator trace with eigenvalues counted according to multiplicity,
rather than the $r$th power of an individual spherical harmonic coefficient.
We measure distinguishability by the total variation distance
$\TV(P,Q)=\sup_B|P(B)-Q(B)|$.
We call the models undetectable when their total variation tends to zero, and strongly detectable when it tends to one.
The latter is equivalent to the existence of tests whose sum of type I and type II errors tends to zero.

The Mao--Wu--Xu spectral conjecture \cite[Section 2.5, Conjecture 1]{MWX} can be written as
\begin{equation}\label{eq:MWX-conjecture}
 \TV\bigl(P_{n,d,W},G(n,p)\bigr)\longrightarrow
 \begin{cases}
 0,& n^3(\tr K^3)^2\longrightarrow0,\\[2pt]
 1,& n^3(\tr K^3)^2\longrightarrow\infty.
 \end{cases}
\end{equation}
This formulation is motivated by the signed triangle statistic. Let
\[
 Z_{ij}=\frac{A_{ij}-p}{\sqrt{p(1-p)}},\qquad
 S_3=\sum_{i<j<k}Z_{ij}Z_{jk}Z_{ki}.
\]
Taking conditional expectations over the edges and integrating the three latent points gives
\[
 \mathbb E_{P_{n,d,W}}S_3=\binom n3\tr K^3,
 \qquad \mathbb E_{G(n,p)}S_3=0,
 \qquad \Var_{G(n,p)}S_3=\binom n3.
\]
Thus $n^3(\tr K^3)^2$ is the order of the squared triangle mean divided by its null variance.
The question is whether the spectral signal captured by this local statistic determines all the information in the observation.
We address this question for general kernel sequences in which $W$ may depend on the dimension.

\subsection{Main result: three explicit detection scales}
Positive and negative contributions to the cubic trace may cancel, whereas the fourth trace is always nonnegative.
Adding the fourth trace is still insufficient: the spherical geometry shared by all vertices produces a further scale determined by
$\tr K^2$ and the dimension $d$.
The following theorem combines the three scales in a single sample-size formula.

\begin{theorem}\label{thm:main}
Fix an integer $L\ge1$ and $p_0\in(0,1/2]$.
Let $n\to\infty$, $d=d_n\ge2$, and let $W=W_n$ be a real polynomial satisfying
\[
 \deg W\le L,\qquad 0\le W\le1,\qquad p_0\le p\le1-p_0.
\]
Let $K$ be the integral operator in \eqref{eq:normalized-operator}, and define
\begin{equation}\label{eq:sample-scale}
 n_*=
 \min\left\{
 \frac{1}{|\tr K^3|^{2/3}},\quad
 \frac{1}{\sqrt{\tr K^4}},\quad
 \frac{d}{\tr K^2}
 \right\}.
\end{equation}
A term with zero denominator is interpreted as $+\infty$. Then
\begin{equation}\label{eq:main-theorem}
 \TV\bigl(P_{n,d,W},G(n,p)\bigr)\longrightarrow
 \begin{cases}
 0,& n/n_*\longrightarrow0,\\[2pt]
 1,& n/n_*\longrightarrow\infty.
 \end{cases}
\end{equation}
The dimension and all polynomial coefficients may vary with $n$.
\end{theorem}

The three quantities in \eqref{eq:sample-scale} are, respectively, the triangle, four-cycle, and geometric scales.
The kernel enters the formula through three operator traces. Given the nonzero eigenvalues $\lambda_1,\ldots,\lambda_r$ of $K$, counted with multiplicity,
one simply substitutes $\tr K^j=\sum_{i=1}^r\lambda_i^j$.

The scale $n_*$ may itself depend on $n$.
The theorem describes the two regimes $n\ll n_*$ and $n\gg n_*$,
where $a\ll b$ means $a/b\to0$ and $a\gg b$ means $b/a\to0$.
It does not assign a universal critical constant to the regime $n\asymp n_*$.
Indeed, Appendix~\ref{app:critical} gives examples with $n/n_*$ tending to a finite positive constant for which strong detection is still possible.

\subsection{What the conjecture misses: cancellation and positive-spectrum counterexamples}
Neither of the last two terms in Theorem~\ref{thm:main} can be omitted for general kernel sequences.
The first family of counterexamples combines the first and second spherical harmonic levels: their cubic traces cancel exactly,
while a signal detectable by four-cycles remains.
The second family requires no spectral cancellation: all nonzero eigenvalues have the same sign and both triangles and four-cycles have insufficient signal,
yet the global geometry is already detectable.

Let $P_{4,d}$ be the degree-four zonal spherical harmonic polynomial normalized by $P_{4,d}(1)=1$.
The next corollary gives the second family and its detection scale in both separated regimes.

\begin{corollary}\label{cor:quartic}
Fix $b\in(0,1)$ and let $W_d=(1+bP_{4,d})/2$, with $d\to\infty$.
All nonzero eigenvalues of its standardized integral operator are positive, and
\begin{equation}\label{eq:quartic-threshold}
 \TV\bigl(P_{n,d,W_d},G(n,1/2)\bigr)\longrightarrow
 \begin{cases}
 0,& n\ll d^5,\\[2pt]
 1,& n\gg d^5.
 \end{cases}
\end{equation}
In particular, throughout $d^5\ll n\ll d^{16/3}$,
\begin{equation}\label{eq:intro-gap}
 n^3(\tr K^3)^2+n^4(\tr K^4)^2\longrightarrow0,
 \qquad \TV\bigl(P_{n,d,W_d},G(n,1/2)\bigr)\longrightarrow1.
\end{equation}
\end{corollary}

The failure of the cubic-trace criterion is therefore not solely due to cancellation between positive and negative eigenvalues.
Nor can this example be detected by a polynomial mean test of any higher but fixed degree:
for every fixed $D$, the orthogonal projection of its likelihood ratio onto polynomials of degree at most $D$ converges to the constant one.
The precise statement and proof are given in Appendix~\ref{app:projection}.
This is a strict separation between detection from the full graph and fixed-degree mean tests, not a general computational complexity lower bound.

\subsection{From triangle thresholds to spectral questions for general kernels}
Geometric detection is first a question about a full distribution: the edge marginals already agree with those of the independent-edge model,
so which dependencies reveal the latent space?
For dense spherical graphs with a threshold connection, Bubeck, Ding, Eldan, and R\'acz \cite[Theorems 1(a), 1(c), and 2]{BDER}
established detection by signed triangles and a matching comparison of full distributions at the scale $d\asymp n^3$.
The role of triangles in this model therefore goes beyond providing a convenient test:
their detection scale matches an impossibility bound that applies to every test.

Soft connection models raise a further question: if each edge is made less sensitive to the inner product,
does this optimality of triangles persist?
Liu and R\'acz \cite{LRlatent} studied Gaussian latent points and monotone inner-product connection functions,
obtaining upper and lower bounds for soft-connection detection and using signed triangles for detection.
In the spherical model, Mao, Wu, and Xu \cite[Theorem 1]{MWX} proved that
a fixed smooth connection function bounded away from zero and one, with $W'(0)\ne0$, has critical dimension $n^{3/4}$,
again attained by signed triangles.
Although threshold connections and nondegenerate smooth connections have different dimension scales,
both agree with the order of $n^3(\tr K^3)^2$.
The cubic trace is thus a natural candidate for a unified description, incorporating the effect of the connection function into one spectral quantity.
The Mao--Wu--Xu conjecture extends this agreement in established cases to general kernels.

The difficulty is that the cubic trace records only a signed spectral sum, rather than all the information across harmonic levels.
For dimension-dependent kernels, the first harmonic level that previously dominated the triangle signal may weaken or disappear,
and cubic contributions from different levels may cancel exactly.
Such sequences lie outside the nondegenerate setting of the fixed smooth-kernel theorem with $W'(0)\ne0$.
Our question concerns the detection boundary in these degenerate cases, rather than the validity of that theorem.

Earlier work already exhibits concrete limitations of triangles.
Bangachev and Bresler \cite{BBcube} proved that four-cycles outperform triangles in a range of parameters for the $L_\infty$ torus model.
In their work on random algebraic graphs, they introduced Fourier-paired connections \cite{BBalgebra},
whose positive and negative Fourier coefficients occur in pairs, making every odd spectral sum vanish.
The same work also gives examples separating statistical detection from low-degree methods.
Thus spectral cancellation as a mechanism for the failure of triangles has precedents.
The further questions here are how fixed low-degree polynomials realize this phenomenon in the spherical inner-product model,
and whether cycle statistics can miss detectable geometric information even without spectral cancellation.

Another directly related work with the same spherical latent space is
the study of Fourier coefficients of high-dimensional geometric graphs by Bangachev and Bresler \cite{BBfourier}.
They analyze low-degree detection for spherical threshold graphs and their Gaussian analogues by estimating signed subgraph moments.
The distinction here concerns both the connection kernel and the desired conclusion: we consider fixed-degree polynomial kernels whose coefficients may vary with the dimension,
and seek impossibility conditions for the full observation distribution.
Fourier estimates for threshold graphs do not directly supply the uniform comparison required for this class of kernels.

We answer these two questions within a single spherical model.
The quadratic cancellation family shows that the cubic trace is insufficient. The pure degree-four positive-spectrum family further shows that, even without cancellation,
the cubic and fourth traces together may miss global geometry.
Beyond these counterexamples, Theorem~\ref{thm:main} proves the converse conclusion for full distributions:
for every dense polynomial sequence of fixed degree, all graph tests fail when all three signals are below their detection levels.
This is the distinction between a uniform impossibility proof and the construction of a single successful statistic.
In the standard sense of low-degree likelihood ratio projections
\cite{KWB}, Appendix~\ref{app:projection} shows that no fixed-degree polynomial mean test recovers the information missed in the positive-spectrum quartic example.

Varying the edge density is a separate issue.
Liu, Mohanty, Schramm, and Yang \cite{LMSY} proved that,
for fixed $\alpha\ge1$ and $p=\alpha/n$, a polylogarithmic dimension already suffices for spherical threshold graphs to be undetectable.
Du, Mao, Sun, Wu, and Xu \cite[Theorem 1]{DMSWX} further proved that, for fixed $\varepsilon>0$,
if $\varepsilon/n\le p\le1-\varepsilon$, $d\ge(1+\varepsilon)n$, and
$d\gg\{nH(p)\}^3$, then the relative entropy of the threshold graph with respect to the independent-edge graph tends to zero,
where $H(p)=-p\log p-(1-p)\log(1-p)$ is binary entropy.
These results show that sparsity can itself change the required geometric comparison; one cannot simply let $p\to0$ in a dense result.
We keep a positive lower bound on the mean edge density, separating spectral degeneracy from sparsity.
The sparse quartic problem in Section~\ref{sec:conclusion} provides a concrete starting point for bringing these two issues together.

\subsection{Proof strategy}
The three detection scales are attained by cycle statistics and a finite partition of the sphere.
The main difficulty lies in the converse: even if every fixed small graph has weak signal,
one must rule out the accumulation of many weak dependencies into detectable information.
We distinguish two ways of adding edges in the expansion of the joint distribution.
Attaching a new vertex enlarges the set of latent positions, at a cost controlled by the spherical spectral gap;
adding an edge between existing vertices increases the cycle rank, at a cost controlled by the squared $L^2$ norm of the kernel.
For graphs of bounded cycle rank, integrating long paths leaves a core of bounded size.
This allows the expansion order to grow with $n$, and ultimately controls the full graph distribution through relative entropy.

Section~\ref{sec:prelim} gives the spherical spectral formulas and uniform moment estimates used in the calculations.
Section~\ref{sec:structure} then constructs examples dominated respectively by triangles, four-cycles, and global geometry,
and proves Corollary~\ref{cor:quartic} and the perturbation boundary for the quadratic cancellation family.
The detection conclusions for these examples follow from the main theorem and are not used in its proof.

The proof of the main theorem has two directions.
Section~\ref{sec:detection} constructs three tests, establishing strong detection when $n\gg n_*$.
Section~\ref{sec:nondetection} proves that the full graph distribution approaches the independent-edge model when $n\ll n_*$,
using, in order, a graph expansion of cumulants, higher-order integration by parts, and relative entropy comparison.
The addition of new vertices is controlled by the spherical spectral gap, and long cycles by operator traces;
together these estimates make the remainder vanish as the expansion order grows.
Section~\ref{sec:conclusion} formulates concrete questions about critical constants, growing-degree tests, and sparse quartic models.
The appendices supply the proof for fixed-degree likelihood ratio projections and the strong-detection phenomenon at a finite geometric scale.

\section{Spectral decomposition and uniform estimates for spherical kernels}\label{sec:prelim}
A fixed polynomial degree has two consequences: higher moments of the kernel are uniformly controlled by its second moment,
and the fourth trace of its integral operator controls the cube of its second trace.
The first handles edges sharing vertices; the second makes the remainder of the higher-order expansion vanish.
All spherical integrals use uniform measure of total mass one, so spectral multiplicities cannot be absorbed into the normalization of integration.

\subsection{Spherical harmonic decomposition}
Let $\mu_d$ be the uniform probability measure on $S^{d-1}$.
Write $\mathcal H_{\ell,d}$ for the restrictions to the sphere of homogeneous harmonic polynomials of degree $\ell$ on $\mathbb R^d$,
and let $M_{\ell,d}$ be its dimension. For $d\ge3$,
\[
 M_{\ell,d}=\binom{d+\ell-1}{\ell}-\binom{d+\ell-3}{\ell-2};
\]
a binomial coefficient with negative lower index is zero. For $d=2$, $M_{0,2}=1$ and
$M_{\ell,2}=2$ for $\ell\ge1$. These spaces are mutually orthogonal,
and every polynomial on the sphere has a unique decomposition into elements of finitely many such spaces.

Choose a real orthonormal basis of $\mathcal H_{\ell,d}$,
$Y_{\ell,1},\ldots,Y_{\ell,M_{\ell,d}}$.
The addition formula defines the normalized polynomial $P_{\ell,d}$:
\begin{equation}\label{eq:addition}
 \sum_{r=1}^{M_{\ell,d}}Y_{\ell,r}(x)Y_{\ell,r}(y)
   =M_{\ell,d}P_{\ell,d}(\langle x,y\rangle),\qquad P_{\ell,d}(1)=1.
\end{equation}
The sum on the left is invariant under simultaneous rotations of $x,y$, and hence depends only on their inner product.
Setting $x=y$ and integrating gives the constant $M_{\ell,d}$ on the right.
Cauchy--Schwarz immediately yields $|P_{\ell,d}(t)|\le1$.
For these standard formulas and the spectral decomposition of the spherical Laplacian used below, see
\cite[Sections 1.2 and 1.4]{DaiXu}; our normalization uses probability measure.

Hereafter let $k$ be a real polynomial of degree at most a fixed integer $L\ge1$,
satisfying
\[
 \int k(\langle x,y\rangle)\,d\mu_d(y)=0,\qquad
 \|k\|_\infty\le b,
\]
where $b$ is fixed. The integral above is independent of $x$.
Write
\[
 k(t)=\sum_{\ell=1}^{L}a_\ell P_{\ell,d}(t),\qquad
 (Tf)(x)=\int k(\langle x,y\rangle)f(y)\,d\mu_d(y).
\]
By \eqref{eq:addition}, $T$ has eigenvalue
$a_\ell/M_{\ell,d}$ on $\mathcal H_{\ell,d}$ and is zero on the remaining spaces. To simplify subsequent calculations, write
\begin{equation}\label{eq:hdt}
 h=\tr T^2=\sum_{\ell=1}^L\frac{a_\ell^2}{M_{\ell,d}},\quad
 \delta=\tr T^3=\sum_{\ell=1}^L\frac{a_\ell^3}{M_{\ell,d}^2},\quad
 \tau=\tr T^4=\sum_{\ell=1}^L\frac{a_\ell^4}{M_{\ell,d}^3}.
\end{equation}
These letters are abbreviations for this section and the proofs; the main theorem is stated directly in terms of operator traces.
If $h=0$, then $k=0$ and the graph distribution is the independent-edge model.
We exclude this case whenever dividing by $h$ or $\tau$ below.

\subsection{Moment comparison at fixed degree}
The following estimate allows both the coefficients and the dimension to vary. It also applies to the integral kernel of $T^r$,
since taking operator powers introduces no new spherical harmonic degrees.

\begin{lemma}\label{lem:moments}
For each fixed positive integer $q$, there is a constant $C_{L,q}$ depending only on $L,q$
such that every real polynomial $p$ of degree at most $L$ satisfies
\[
 \mathbb E|p(\langle X,Y\rangle)|^{2q}
 \le C_{L,q}\bigl(\mathbb E p(\langle X,Y\rangle)^2\bigr)^q,
\]
where $X,Y$ are independent with law $\mu_d$, uniformly for all $d\ge2$.
Moreover, for the bounded centered kernel $k$ above,
\begin{equation}\label{eq:h-cube}
 h^3\le C_{L,b}\tau.
\end{equation}
\end{lemma}

\begin{proof}
Let $Z=\sqrt d\,\langle X,Y\rangle$. Rotational invariance and the even-moment formula for a spherical coordinate give
\[
 \mathbb E Z^{2r}=(2r-1)!!\prod_{j=0}^{r-1}\frac{d}{d+2j},
 \qquad \mathbb E Z^{2r+1}=0.
\]
The even-moment formula also follows from the independence of the direction and length of a standard Gaussian vector:
divide the $2r$th moment of its first coordinate by the $2r$th moment of its length.
Thus moments of each fixed order are uniformly bounded and converge to the corresponding standard Gaussian moments.

Write $p(Z/\sqrt d)=\sum_{j=0}^L c_jZ^j$.
The matrix $(\mathbb E Z^{i+j})_{0\le i,j\le L}$ converges to the positive definite Gaussian moment matrix.
For each fixed $d\ge2$, the law of $Z$ has a positive density on an interval, so this matrix is also positive definite.
First considering all sufficiently large $d$ and then taking a minimum over the finitely many remaining dimensions gives
\[
 \mathbb E p(Z/\sqrt d)^2\ge c_L\sum_{j=0}^L c_j^2.
\]
On the other hand, the inequality for a finite sum and the uniform moment bounds give
\[
 \mathbb E\left|\sum_{j=0}^L c_jZ^j\right|^{2q}
 \le C_{L,q}\left(\sum_{j=0}^Lc_j^2\right)^q.
\]
Combining the two inequalities proves the first assertion.

To prove \eqref{eq:h-cube}, we first bound the spherical harmonic coefficients.
Starting from $P_{0,d}=1$ and $P_{1,d}=t$, the recurrence
\[
 (\ell+d-2)P_{\ell+1,d}
  =(2\ell+d-2)tP_{\ell,d}-\ell P_{\ell-1,d},
 \qquad \ell\ge1,
\]
shows that, for fixed $\ell$, the monomial coefficients of $P_{\ell,d}$ converge to those of $t^\ell$.
Consequently, the Gram matrix of $P_{1,d},\ldots,P_{L,d}$ in the Lebesgue space
$L^2([-1,1])$ has smallest eigenvalue bounded uniformly away from zero:
the limiting matrix is positive definite, and the finitely many remaining matrices are positive definite because the polynomial degrees are distinct.
The bound $\|k\|_\infty\le b$ therefore yields $\sum_\ell a_\ell^2\le C_Lb^2$.
Set $h_\ell=a_\ell^2/M_{\ell,d}$ and omit terms with $a_\ell=0$. Then
\[
 \tau=\sum_\ell\frac{h_\ell^3}{a_\ell^2}
 \ge c_{L,b}\sum_\ell h_\ell^3
 \ge \frac{c_{L,b}}{L^2}\left(\sum_\ell h_\ell\right)^3.
\]
This proves the claim.
\end{proof}

Let $\|T\|_{\mathrm{op}}$ denote the operator norm on $L^2(\mu_d)$.
Elementary inequalities for finite spectral sums give
\begin{equation}\label{eq:spectral-basic}
 \delta^2\le h\tau,\qquad
 \tr T^6\le\tau^{3/2},\qquad
 \|T\|_{\mathrm{op}}^2\le\sqrt\tau.
\end{equation}
For a graph with $n$ vertices, set
\begin{equation}\label{eq:signals}
 A_3=n^3\delta^2,\qquad A_4=n^4\tau^2,\qquad G=\frac{nh}{d},
 \qquad \rho=n\|T\|_{\mathrm{op}}^2.
\end{equation}
In particular, $\rho\le A_4^{1/4}$, and when $A_4\le1$,
$h\le C_{L,b}n^{-2/3}$.
These quantities will be used to estimate cycle statistics, global geometry, and long paths.

\section{When the three scales cannot replace one another}\label{sec:structure}
The three-scale theorem gives more than a list of sufficient conditions.
At low degrees, some scales can indeed be controlled by the others,
but cubic kernels already exhibit an isolated geometric signal; at degree four this is possible even without spectral cancellation.
We then return to the quadratic cancellation family and determine both separated regimes under perturbation.

\subsection{Restrictions imposed by degree and spectral signs}
Spectral rank explains the special role of low-degree kernels.
Let $r$ be the number of nonzero eigenvalues of $T$, counted with multiplicity, and set
$m_3=\sum_i|\lambda_i|^3$. H\"older's inequality and finite-dimensional norm comparison yield
\begin{equation}\label{eq:rank-bounds}
 h^2\le r\tau,\qquad h^3\le r m_3^2,\qquad \tau^3\le m_3^4.
\end{equation}
For example, the second inequality follows from
$\sum_i|\lambda_i|^2\le r^{1/3}(\sum_i|\lambda_i|^3)^{2/3}$; 
the third follows because the $\ell^4$ norm is at most the $\ell^3$ norm.
Since $r\le\sum_{\ell=1}^LM_{\ell,d}\le C_Ld^L$, these inequalities imply the following.

\begin{proposition}\label{prop:absorption}
If $L\le2$, then $G\le C_LA_4^{1/4}$.
If all nonzero eigenvalues have the same sign and $L\le3$, then
$G\le C_LA_3^{1/3}$, $A_4\le A_3^{4/3}$.
\end{proposition}
\begin{proof}
In the first case,
$G=nh/d\le(\sqrt r/d)n\sqrt\tau\le C_LA_4^{1/4}$.
In the second case, $m_3=|\delta|$, so
$G\le(r^{1/3}/d)n|\delta|^{2/3}\le C_LA_3^{1/3}$.
Finally, $\tau^3\le\delta^4$ is equivalent to $A_4\le A_3^{4/3}$.
\end{proof}
Thus the separated detection boundary for quadratic kernels requires only the triangle and four-cycle scales.
For same-sign spectra of degree at most three, the cubic trace determines the zero and infinite regimes in Theorem~\ref{thm:main}.
The latter statement concerns the detection boundary; it does not assert that triangles are optimal along every detectable sequence.

\subsection{Three examples with different dominant scales}
Throughout this subsection take $b=1/4$ and connection probability $(1+k)/2$.
The standardized kernel is then $k$ itself.
Write $u\asymp v$ when $u/v$ lies between two positive constants independent of $n,d$.
In each of the following sequences, one signal tends to infinity and the other two tend to zero.

First, take $k=bP_{1,d}$. Since $M_{1,d}=d$,
$h=b^2/d$, $\delta=b^3/d^2$, $\tau=b^4/d^3$.
Let $d=q^5$ and $n=q^7$, with integers $q\to\infty$. Then
\[
 A_3\asymp q,\qquad A_4\asymp q^{-2},\qquad G\asymp q^{-3}.
\]
Only the triangle signal diverges.

Next, let $D=M_{2,d}=(d-1)(d+2)/2$ and take
\[
 k=aP_{1,d}-bP_{2,d},\qquad a=b(d/D)^{2/3}.
\]
For $d\ge3$, $a\le b$, hence $|k|\le1/2$.
The cubic traces of the two spectral levels cancel exactly, while
\[
 \delta=0,\qquad h\asymp d^{-2},\qquad \tau\asymp d^{-17/3}.
\]
With $d=q^{12}$ and $n=q^{35}$,
$A_3=0$, $A_4\asymp q^4$, $G\asymp q^{-1}$.
Thus exact cubic-trace cancellation in this quadratic family leaves four-cycles as the only diverging signal.

The third example moves the cancellation to the second and third harmonic levels.
Let $R=M_{2,d}$ and $S=M_{3,d}=d(d-1)(d+4)/6$, and take
\[
 k=aP_{2,d}-bP_{3,d},\qquad a=b(R/S)^{2/3}.
\]
Again $|k|\le1/2$ and $\delta=0$, while
$h\asymp d^{-3}$, $\tau\asymp d^{-26/3}$.
Setting $d=q^6$ and $n=q^{25}$ gives
\[
 A_3=0,\qquad A_4\asymp q^{-4},\qquad G\asymp q.
\]
The simultaneous vanishing of the triangle and four-cycle signals therefore does not imply that the full graph is undetectable.
The common geometric constraints on all vertices are decisive here.

\subsection{A quartic kernel without spectral cancellation}
A stronger phenomenon arises from a single harmonic level. In this subsection, $b\in(0,1)$ is any fixed constant.
For $d\ge3$, the normalized degree-four zonal polynomial is
\[
 P_{4,d}(t)=\frac{(d+2)(d+4)t^4-6(d+2)t^2+3}{(d-1)(d+1)},
 \qquad M=M_{4,d}=\frac{d(d-1)(d+1)(d+6)}{24}.
\]
Take $k=bP_{4,d}$. By the addition formula, all its nonzero eigenvalues equal the positive number $b/M$,
with multiplicity $M$, so
\begin{equation}\label{eq:quartic-traces}
 h=\frac{b^2}{M},\qquad \delta=\frac{b^3}{M^2},\qquad
 \tau=\frac{b^4}{M^3}.
\end{equation}
There is no sign cancellation in the cubic trace, but the three sample-size scales are
\[
 |\delta|^{-2/3}\asymp d^{16/3},\qquad
 \tau^{-1/2}\asymp d^6,\qquad d/h\asymp d^5.
\]
The geometric signal is therefore the first to appear.
Theorem~\ref{thm:main} gives impossibility when $n/d^5\to0$
and strong detection when $n/d^5\to\infty$, proving Corollary~\ref{cor:quartic}.
In particular, with $d=q^4$ and $n=q^{21}$,
\begin{equation}\label{eq:quartic-gap}
 A_3\asymp q^{-1},\qquad A_4\asymp q^{-12},\qquad G\asymp q.
\end{equation}
A spectrum of one sign is not enough for cycle signals to determine detection;
the dimension of the space supporting the spectrum and its spherical realization also matter.

\subsection{The perturbation boundary for quadratic cancellation}
The quadratic family also exhibits a continuous transition between triangles and four-cycles.
In this subsection fix $b=1/4$ again. With $D=(d-1)(d+2)/2$, let
\begin{equation}\label{eq:perturbed-kernel}
 k_{d,\varepsilon}=a(1+\varepsilon)P_{1,d}-bP_{2,d},\qquad
 a=b(d/D)^{2/3},\qquad |\varepsilon|\le\tfrac12.
\end{equation}
For $d\ge3$, $\|k_{d,\varepsilon}\|_\infty\le5/8$, so this always defines a valid dense graph model.

\begin{corollary}\label{cor:quadratic}
Let $d=d_n\ge3$, $|\varepsilon_n|\le1/2$, and set
\[
 d_*(n,\varepsilon_n)=
 \max\{n^{6/17},\ n^{3/8}|\varepsilon_n|^{1/4}\}.
\]
For the graph defined by \eqref{eq:perturbed-kernel},
its total variation distance from $G(n,1/2)$ tends to zero when $d/d_*\to\infty$,
and tends to one when $d/d_*\to0$.
\end{corollary}
\begin{proof}
The exact formulas for the two spectral levels give
\[
 \delta=\frac{b^3}{D^2}\{(1+\varepsilon)^3-1\},\qquad
 h\asymp d^{-2},\qquad \tau\asymp d^{-17/3},
 \qquad \delta^2\asymp\varepsilon^2d^{-8}.
\]
The comparison constants are uniform for $|\varepsilon|\le1/2$; at $\varepsilon=0$, both sides of the last comparison are zero.
The last comparison follows from $(1+\varepsilon)^3-1=\varepsilon(3+3\varepsilon+\varepsilon^2)$,
whose second factor is bounded away from zero on this interval.
By Proposition~\ref{prop:absorption}, $G$ is controlled by the four-cycle scale.
Moreover, $A_3\asymp(n^{3/8}|\varepsilon|^{1/4}/d)^8$ and
$A_4\asymp(n^{6/17}/d)^{34/3}$.
The two signals tend to zero simultaneously, or their maximum tends to infinity, precisely under the two stated separation conditions.
The conclusion follows from Theorem~\ref{thm:main}.
\end{proof}
At exact cancellation, the boundary is $d\asymp n^{6/17}$;
the two terms have the same order when $|\varepsilon|\asymp n^{-3/34}$.
Thus cancellation affects detection beyond a single isolated parameter value:
it changes the dominant statistic throughout a neighborhood that shrinks with sample size.

\section{Detection by cycle statistics and geometric tests}\label{sec:detection}
Triangles and four-cycles use relations among only a few vertices; geometric detection compares the positions of all vertices at once.
We construct these three tests and show that their total error can tend to zero whenever one of the corresponding signals diverges.
To simplify the variance calculations, first consider edge density $1/2$ and conditional edge probability
$(1+k(\langle X_i,X_j\rangle))/2$, where $\|k\|_\infty\le b<1$.
Denote its law by $P_k$ and the comparison law by $Q=G(n,1/2)$.
General dense edge probabilities are treated in Section~\ref{sec:transfer}.

\subsection{Triangles and four-cycles}
We use the standard signed cycle statistics for spherical detection \cite{BDER,MWX}.
Let $Y_{ij}=2A_{ij}-1$. Under $Q$, these are independent symmetric signs;
conditional on the latent positions, their means are $k(\langle X_i,X_j\rangle)$.
For a simple undirected cycle $C$, considered without a starting point or orientation, write $Y_C=\prod_{e\in C}Y_e$.
Let $S_3,S_4$ be the sums of $Y_C$ over all triangles and four-cycles, respectively, and set
$N_3=\binom n3$, $N_4=3\binom n4$.
Integrating the vertices of a cycle successively gives
\[
 \mathbb E_{P_k}S_3=N_3\delta,\qquad
 \mathbb E_{P_k}S_4=N_4\tau.
\]
Products of signs over distinct edge sets are orthogonal under $Q$, so
$\mathbb E_QS_r=0$, $\Var_QS_r=N_r$ for $r=3,4$.

\begin{lemma}\label{lem:cycle-detection}
If $A_4\to\infty$, the four-cycle statistic gives strong detection.
If $A_3\to\infty$ and $A_4\le A_3$, the triangle statistic gives strong detection.
\end{lemma}
\begin{proof}
Two triangles with no common vertex are independent. If they share exactly one vertex, they are conditionally independent given that vertex,
and both conditional means equal the constant $\delta$, independent of its position; hence their covariance is zero.
If two distinct triangles share an edge, $Y_e^2=1$ reduces their product to a four-cycle.
There are six such pairs on each four-vertex set. Thus
\begin{equation}\label{eq:triangle-var}
 \Var_{P_k}S_3=N_3(1-\delta^2)
       +12\binom n4(\tau-\delta^2)
 \le C(n^3+n^4\tau).
\end{equation}
When $\delta\ne0$, reject on the event
$\operatorname{sgn}(\delta)S_3>N_3|\delta|/2$.
Applying Chebyshev's inequality under both laws bounds the total error by
\[
 C\left(\frac1{A_3}+\frac{\sqrt{A_4}}{nA_3}\right).
\]
This tends to zero under the triangle assumptions of the lemma.

The four-cycle variance has more overlap types. Let $q(x,y)$ and $r(x,y)$ be the integral kernels of $T^2$ and $T^3$, respectively, and set
\[
 s_6=\tr T^6,\qquad u=\mathbb E q(X,Y)^4,\qquad
 v=\mathbb E[k(\langle X,Y\rangle)r(X,Y)q(X,Y)^2].
\]
The table lists all overlap types of two distinct four-cycles that can have nonzero covariance.
Counts refer to unordered pairs on a fixed union of vertices; the mean is the expectation of the product of the two cycle statistics.
\begin{center}
\begin{tabular}{clcc}
\toprule
Union size&Overlap&Product mean&Count\\
\midrule
6&One shared edge&$s_6$&180\\
6&Shared vertices nonadjacent in both&$u$&45\\
6&Shared vertices adjacent in one only&$v$&180\\
5&One shared edge&$\delta^2$&60\\
5&Two shared edges&$\tau$&30\\
4&Two shared edges&$\tau$&3\\
\bottomrule
\end{tabular}
\end{center}
For six vertices, there are 15 choices of the two shared vertices and 3 ways to split the remaining four into two pairs.
On a fixed set of four vertices, there are two four-cycles in which the shared vertices are adjacent and one in which they are not.
The three counts are therefore $15\cdot3\cdot4$, $15\cdot3$, and $15\cdot3\cdot4$.
For five vertices, there are 10 choices of the shared triple; each four-cycle induces a two-edge path on this triple.
There are 3 ordered choices of identical paths and 6 of distinct paths, giving 30 and 60, respectively.
On four vertices there are three four-cycles, giving three unordered pairs.
With fewer than two common vertices, the covariance is zero as in the triangle case.

After the shared edges are canceled, the three six-vertex types become a six-cycle and four paths with common endpoints of lengths
$(2,2,2,2)$ and $(1,3,2,2)$, respectively.
The five-vertex type with one shared edge becomes two triangles meeting at one vertex;
the other two types become four-cycles. This also proves the product means in the table.
Consequently,
\begin{align}
 \Var_{P_k}S_4={}&N_4(1-\tau^2)\notag\\
 &+2\binom n6\{180(s_6-\tau^2)+45(u-\tau^2)+180(v-\tau^2)\}\notag\\
 &+2\binom n5\{60(\delta^2-\tau^2)+30(\tau-\tau^2)\}
       +6\binom n4(\tau-\tau^2).\label{eq:four-var}
\end{align}
By Lemma~\ref{lem:moments}, H\"older's inequality, and \eqref{eq:spectral-basic},
\[
 u\le C_L\tau^2,\qquad
 |v|\le C_L\sqrt h\sqrt{s_6}\,\tau\le C_L\sqrt h\,\tau^{7/4}.
\]
Dropping negative terms in \eqref{eq:four-var} and using $h\le b^2\le1$,
we obtain, for $\tau>0$ and all sufficiently large $n$,
\begin{equation}\label{eq:four-ratio}
 \frac{\Var_{P_k}S_4+\Var_QS_4}{(N_4\tau)^2}
 \le C_L\left\{
 \frac1{A_4}+\frac1{nA_4^{1/4}}+\frac1{n^2}
 +\frac1{n^{3/2}A_4^{1/8}}+\frac1{nA_4^{1/2}}
 \right\}.
\end{equation}
For example, after division by $n^8\tau^2$, the terms $n^6s_6$ and $n^6|v|$
are bounded by $1/(nA_4^{1/4})$ and $1/(n^{3/2}A_4^{1/8})$, respectively;
both $n^5\delta^2$ and $n^5\tau$ are controlled by the last term.
Rejecting when $S_4>N_4\tau/2$ and applying Chebyshev's inequality proves the claim.
\end{proof}

\subsection{Preserving geometric information by a finite partition}
Geometric detection does not require knowledge of the true latent positions. We replace positions by finitely many spherical regions
and search over all assignments of vertices to regions for one consistent with the observed edges.
The key is that $\exp(C_Ld)$ regions suffice to preserve a fixed fraction of the squared $L^2$ norm of a fixed-degree kernel.

\begin{lemma}\label{lem:partition}
There is a constant $c_L$ depending only on $L$ such that every nonzero kernel $k$ as above admits a measurable partition
$S^{d-1}=B_1\cup\cdots\cup B_J$ with $\log J\le c_Ld$.
Let
\[
 g(x,y)=\mathbb E[k(\langle X,Y\rangle)\mid X\in B_a,\ Y\in B_b],
 \qquad x\in B_a,\ y\in B_b,
\]
with value zero on regions of measure zero. Then
$\|g\|_\infty\le b$ and $\|k-g\|_{L^2(\mu_d\otimes\mu_d)}^2\le h/32$.
\end{lemma}
\begin{proof}
The addition formula gives
\[
 \|k(\langle x,\cdot\rangle)-k(\langle z,\cdot\rangle)\|_2^2
 =2h\sum_{\ell=1}^L w_\ell[1-P_{\ell,d}(\langle x,z\rangle)],
 \quad w_\ell=\frac{a_\ell^2}{M_{\ell,d}h},\quad \sum_\ell w_\ell=1.
\]
For polynomials of fixed degree, the derivative norm is controlled by the supremum norm:
interpolate at $L+1$ fixed distinct nodes and differentiate to obtain
$\|p'\|_\infty\le C_L\|p\|_\infty$.
The preceding squared distance is therefore at most $C_Lh\|x-z\|^2$.

Take a maximal $\varepsilon$-separated set. The Euclidean balls of radius $\varepsilon/2$ about its points are disjoint
and contained in a ball of radius $1+\varepsilon/2$. Comparing volumes bounds the number of points by
$(1+2/\varepsilon)^d$. Maximality also makes this an $\varepsilon$-net.
Partition the sphere by nearest net points, resolving ties in a fixed order.
Replacing each row by the row at its net point incurs squared error at most $C_Lh\varepsilon^2$.

Let $P_x,P_y$ be the orthogonal projections given by conditional averaging over this partition in the respective coordinates.
Conditional expectation is the best approximation in $L^2$, so
$\|k-P_xk\|_2^2,\|k-P_yk\|_2^2\le C_Lh\varepsilon^2$.
The projections commute, and
\[
 k-P_xP_yk=(I-P_x)k+P_x(I-P_y)k.
\]
The two terms on the right are orthogonal. The total error is therefore at most $2C_Lh\varepsilon^2$.
Choose $\varepsilon$ sufficiently small, depending only on $L$.
Then $P_xP_yk=g$, and conditional averaging preserves the supremum bound.
\end{proof}

\begin{lemma}\label{lem:scan}
Suppose a symmetric block kernel $g$ satisfies the error bound of Lemma~\ref{lem:partition},
with $J$ regions. Let $N=\binom n2$.
There is a constant $C_b$ depending only on $b$ and a graph test with total error at most
\begin{equation}\label{eq:scan-error}
 \exp\{n\log J-Nh/16\}+\frac{C_b}{nh}.
\end{equation}
In particular, strong detection holds if $G\to\infty$.
\end{lemma}
\begin{proof}
Write $g_{ab}$ for the value on blocks $a,b$. For each assignment
$z=(z_1,\ldots,z_n)\in\{1,\ldots,J\}^n$, let
\[
 F_z=\sum_{i<j}\left(g_{z_i z_j}Y_{ij}-\frac12 g_{z_i z_j}^2\right).
\]
Reject $Q$ when $\max_zF_z\ge Nh/16$.
Under $Q$, $\mathbb E_Qe^{F_z}\le1$, since
$\cosh(t)\le e^{t^2/2}$. Markov's inequality and a union bound over the $J^n$ assignments
give the first term in \eqref{eq:scan-error}.

Under $P_k$, consider the assignment given by the regions containing the true positions. Set
\[
 V=\sum_{i<j}k(\langle X_i,X_j\rangle)^2,\qquad
 D=\sum_{i<j}(k(\langle X_i,X_j\rangle)-g(X_i,X_j))^2.
\]
Every row has squared $L^2$ norm $h$, so the squared kernel terms on different edges are uncorrelated even when they share a vertex.
Thus $\mathbb EV=Nh$ and $\Var V\le Nb^2h$.
Also, $\mathbb ED\le Nh/32$, although the error kernel need not have constant row integrals.
If $U_{ij}=(k_{ij}-g_{ij})^2$, then
$0\le U_{ij}\le4b^2$, $\mathbb EU_{ij}^2\le4b^2\mathbb EU_{ij}\le C_bh$.
There are $O(n^3)$ pairs of edges sharing a vertex. Cauchy--Schwarz gives
$\Var D\le C_bn^3h$. Outside an event of probability at most $C_b/(nh)$, therefore,
\[
 Nh/2\le V\le2Nh,\qquad D\le Nh/4.
\]
Conditional on the positions, the score of the true assignment has mean $(V-D)/2\ge Nh/8$
and variance at most
$\sum g(X_i,X_j)^2\le2V+2D\le(9/2)Nh$.
Conditional Chebyshev bounds the probability that it falls below $Nh/16$ by $C_b/(Nh)$.
Combining the bounds proves \eqref{eq:scan-error}.
When $G=nh/d\to\infty$, we have $nh\to\infty$ and
$n\log J/(Nh)\le C_L/G\to0$, so both error terms vanish.
\end{proof}

\begin{proposition}\label{prop:detect}
If $\max\{A_3,A_4,G\}\to\infty$, then $\TV(P_k,Q)\to1$.
\end{proposition}
\begin{proof}
For each $n$, select the largest of the three signals, breaking ties in a fixed order.
If $A_3$ is selected, then $A_4\le A_3$, and we use the triangle test from Lemma~\ref{lem:cycle-detection}.
If $A_4$ is selected, use the four-cycle test; if $G$ is selected, use Lemma~\ref{lem:scan}.
Along every branch selected infinitely often, the corresponding signal diverges, so the error tends to zero.
There are only three branches, and hence the error of the entire sequence tends to zero.
\end{proof}

\section{Impossibility: from edge dependence to the full distribution}\label{sec:nondetection}
We prove the impossibility direction of Theorem~\ref{thm:main}.
The aim is to control the joint law of all $\binom n2$ edges when all three sample-size scales dominate $n$.
Computing a few cycle moments is insufficient; we need an expansion whose order grows with $n$ while its remainder remains controlled.

The proof has three steps. First, organize cumulants by their multigraphs of edge occurrences
and control the bounded-cycle-rank part by suppressing long paths.
Next, use higher-order integration by parts on the product of spheres to control the remaining coefficients and the remainder at growing order.
Finally, prove a finite-order relative entropy inequality and return to Bernoulli graphs through an edgewise random map.
For the first two steps, $k$ is centered, has degree at most a fixed $L$, and satisfies $\|k\|_\infty\le b<1$;
the last step handles normalization and general dense edge probabilities.
\subsection{The graph expansion of cumulants}\label{sec:cumulants}
An impossibility proof must control the joint law of all edges, rather than only finitely many cycles.
We first treat edge collections whose size may grow but whose complexity remains bounded.
Their long paths can be summed by operator convolution, leaving a finite part controlled by the cubic and fourth traces.
The next subsection treats the remaining edge collections with a different estimate.

\subsubsection{Cumulants and graph support}
For real random variables $Z_1,\ldots,Z_m$ with all moments, define their joint cumulant by
\begin{equation}\label{eq:cumulant-def}
 \cum(Z_1,\ldots,Z_m)
 =\sum_{\pi\in\mathcal P([m])}(-1)^{|\pi|-1}(|\pi|-1)!
      \prod_{B\in\pi}\mathbb E\prod_{j\in B}Z_j.
\end{equation}
Here $\mathcal P([m])$ is the set of partitions of $[m]=\{1,\ldots,m\}$,
and $|\pi|$ is the number of blocks. A cumulant vanishes if the variables split into two nonempty independent groups.
For bounded variables, for instance, this follows directly from the logarithm of the moment generating function:
independence splits the logarithm into two functions with no mixed variables, and cumulants are its mixed derivatives.
Formula~\eqref{eq:cumulant-def} follows by expanding the logarithm.

Let $\theta_{ij}=k(\langle X_i,X_j\rangle)$.
Allowing a variable to occur repeatedly in a cumulant, represent its occurrences by a loopless multigraph $H$
whose vertex set is a subset of $[n]$. Each distinct edge $f$ has multiplicity $m_f\ge1$.
Write $m=\sum_fm_f$ for the total number of edge occurrences, $v$ for the number of vertices used, and set
\[
 C_H=\cum(\theta_f:\ f\text{ occurs }m_f\text{ times}),\qquad t=m-v+1.
\]
The edge occurrences are distinguished in \eqref{eq:cumulant-def},
but $H$ itself is determined only by its vertex labels and edge multiplicities.
For connected graphs, $t$ is the multigraph cycle rank. Trees have $t=0$;
attaching a new vertex by one edge preserves $t$, while adding an edge between existing vertices increases it by one.

\begin{lemma}\label{lem:cut}
If the underlying simple graph of $H$ is disconnected or has a cut vertex, then $C_H=0$.
Also, $C_H=0$ if $H$ has a vertex of degree one, counting multiplicity.
Throughout, degrees count edge multiplicities, and a cut vertex is a vertex whose deletion increases the number of connected components.
\end{lemma}
\begin{proof}
Disconnected support splits the variables into independent groups. For a cut vertex $x$, split the remaining components into two nonempty sides,
and assign each edge variable to its side. Conditional on $X_x$, the two vectors of variables are independent.
For any spherical points $u,u'$, choose an orthogonal transformation sending $u$ to $u'$.
Applying it simultaneously to all latent points on one side preserves their uniform law and all inner products.
Thus the entire conditional law on each side, and not just its mean, is independent of $X_x$.
The two sides are therefore unconditionally independent, and the cumulant vanishes.

If $x$ has degree one, then in every partition $\pi$,
the unique moment factor containing its incident edge vanishes upon integration over $X_x$.
Every term in \eqref{eq:cumulant-def} is consequently zero.
\end{proof}

\subsubsection{Summation at fixed cycle rank}
The next lemma allows an arbitrarily large number of edges $m$, fixing only an upper bound on the cycle rank.
The factorial weight will arise naturally from the higher-order expansion in the next subsection.
For example, join two endpoints by three internally vertex-disjoint paths
of lengths $r_1,r_2,r_3\ge2$.
There are $m=r_1+r_2+r_3$ edges and $v=m-1$ vertices, so the cycle rank is always two.
Integrating the internal vertices of each path replaces it by the corresponding operator convolution kernel.
The paths can be arbitrarily long, but the remaining structure always has two endpoints and three connections.
We use the same decomposition for general graphs of bounded cycle rank.

\begin{lemma}\label{lem:core-sum}
Fix an integer $t_0\ge1$ and a real number $w\ge1$. For $A_4$ sufficiently small,
\begin{equation}\label{eq:core-sum}
 \sum_{\substack{H\text{ connected}\\m\ge3,\ t\le t_0}}
       w^{m-1}\frac{m\,C_H^2}{\prod_fm_f!}
 \le C_{L,b,t_0,w}\bigl(A_3+A_4^{1/6}\bigr).
\end{equation}
The sum is over distinct labeled multigraphs with vertices in $[n]$ and no isolated vertices,
without an additional ordering of the edge occurrences.
\end{lemma}
\begin{proof}
By Lemma~\ref{lem:cut}, only supports without cut vertices need be considered.
First separate two cases with no nontrivial core.

If the support has only two vertices, then $m=t+1\le t_0+1$.
By \eqref{eq:cumulant-def} and Lemma~\ref{lem:moments},
$|C_H|\le C_{L,t_0}h^{m/2}$.
Indeed, each moment factor of size $r$ is at most $C_{L,t_0}h^{r/2}$,
and the block sizes sum to $m$.
Since $h\le b^2$ and $3\le m\le t_0+1$, we have $h^m\le C_{b,t_0}h^3$, so the total contribution is at most
$C n^2h^3\le C n^2\tau=C\sqrt{A_4}$.

If all degrees equal two and there are at least three vertices,
then $H$ is a simple cycle. Every nontrivial partition in the cumulant
puts two adjacent edges in different blocks; one moment factor then has a degree-one vertex and vanishes.
Hence $C_H=\tr T^m$. Triangles contribute $O_w(A_3)$ in total.
For $m\ge4$, the finite spectral sum gives
\[
 |\tr T^m|\le \tau\|T\|_{\mathrm{op}}^{m-4},
 \qquad n^m|\tr T^m|^2\le A_4\rho^{m-4}.
\]
There are at most $n^m$ labeled $m$-cycles, so these cycles contribute at most
\[
 C_w A_4\sum_{m\ge4}m(w\rho)^{m-4}\le C_wA_4
\]
provided $w\rho\le1/2$.

Now suppose the support has at least three vertices and is not one of the cycles above.
Suppress the degree-two vertices successively. Since the support has no cut vertex,
each such vertex is incident to two distinct edges of multiplicity one: if it were joined only to a single neighbor by a repeated edge,
that neighbor would be a cut vertex. Call the remaining vertices core vertices, and denote their number by $v_0$.
There are $R_0$ direct edges between core vertices, counted with multiplicity, and $p$ further paths
of length at least two. Let their lengths be $r_1,\ldots,r_p$,
and set $z_i=r_i-2$, $z=\sum_i z_i$, and $B=R_0+p$.
Then
\begin{equation}\label{eq:core-counts}
 v=v_0+p+z,\qquad m=R_0+2p+z,\qquad
 t=B-v_0+1,\qquad 3v_0\le2B.
\end{equation}
Thus $v_0\le2(t_0-1)$ and $B\le3(t_0-1)$.
This case is empty when $t_0=1$.

Each suppressed path has distinct endpoints. Otherwise it meets the rest of the graph at only one core vertex,
which would be a cut vertex; if there were no other vertices, the original graph would be a cycle.
Adding a repeated edge to a cycle creates two distinct core endpoints, so this gives no exception.
Thus every convolution kernel below is evaluated at two independent spherical points.

In every nonzero term of \eqref{eq:cumulant-def}, all edges of a given path belong to one block;
otherwise a degree-one vertex appears at a separation point. Each long path may therefore be treated as one unit,
and each occurrence of a direct edge as another unit.
Nonzero partitions are partitions of these $B$ units, so their number and the coefficients $(|\pi|-1)!$ depend only on $t_0$.
Within each moment factor, integrate the internal vertices of a path of length $r$ to replace it by the integral kernel of $T^r$.
This is still an inner-product polynomial of degree at most $L$, and
\[
 \|T^r\|_{\mathrm{HS}}^2
   =\sum_\ell M_{\ell,d}(a_\ell/M_{\ell,d})^{2r}
   \le \tau\|T\|_{\mathrm{op}}^{2r-4},\qquad r\ge2.
\]
Apply H\"older's inequality to the finitely many factors on the core and then Lemma~\ref{lem:moments}
to obtain
\begin{equation}\label{eq:core-moment}
 C_H^2\le C_{L,t_0}h^{R_0}\tau^p\|T\|_{\mathrm{op}}^{2z}.
\end{equation}
H\"older's inequality does not require independence of the factors. Their individual endpoints are distinct,
which ensures that their squared $L^2$ norms equal the corresponding squared Hilbert--Schmidt norms.

For a fixed core and length vector, first number the core vertices and orient the paths,
then fill in the labels of the core vertices and internal path vertices in order.
The bounded core size gives at most $C_{t_0}n^v$ such encodings.
Every graph has at least one encoding; allowing repeated labels only increases the upper bound.
No additional factor of $m!$ or $v!$ is needed.
The weight satisfies $m/\prod_fm_f!\le m$.
When $A_4\le1$, substituting \eqref{eq:h-cube} into \eqref{eq:core-moment} gives
\begin{align}
 n^vC_H^2
 &\le C_{L,b,t_0}
  n^{\,v_0-2R_0/3-p} A_4^{R_0/6+p/2}\rho^z\notag\\
 &\le C_{L,b,t_0}A_4^{1/6}\rho^z.\label{eq:core-small}
\end{align}
For the second inequality, $v_0\le2(R_0+p)/3$, so the exponent of $n$ is at most $-p/3$;
also $R_0+p\ge1$, so the exponent of $A_4$ is at least $1/6$.

Finally, sum over path lengths. With $a=R_0+2p$ and $x=w\rho$,
\[
 \sum_{z_1,\ldots,z_p\ge0}(a+\textstyle\sum_i z_i)
       w^{a+\sum_i z_i-1}\rho^{\sum_i z_i}
 =w^{a-1}\left\{\frac{a}{(1-x)^p}
             +\frac{px}{(1-x)^{p+1}}\right\}.
\]
This identity follows from the geometric series and its derivative. When $p=0$, the sum has one empty length vector and equals $w^{a-1}a$.
For $w\rho\le1/2$, the right-hand side depends only on $t_0,w$.
There are finitely many core types. Adding the repeated-edge and cycle contributions proves \eqref{eq:core-sum}.
\end{proof}

\subsection{Higher-order integration by parts on the sphere}\label{sec:stein}
The preceding subsection controls cumulants of bounded cycle rank. We now show why terms of large cycle rank also vanish.
The key distinction is between two ways of adding an edge: attaching a new vertex costs an amount controlled by $nh/d$,
whereas adding an edge between existing vertices has a cost that vanishes with the squared $L^2$ norm of the kernel.
The first bound must be independent of the expansion order; the spherical spectral gap and conditional centering at leaves provide exactly this property.
In this subsection and the next, $k$ is centered, has degree at most a fixed $L$, and satisfies $|k|\le b<1$.

\subsubsection{The recursion and its support}
Let $E=\binom{[n]}2$, $N=|E|$, and regard $\theta=(\theta_e)_{e\in E}$ as a random vector.
Distinct edges have zero covariance even when they share an endpoint, as one may first integrate the other endpoint.
Thus $\mathbb E\theta=0$ and $\operatorname{Cov}(\theta)=hI_N$.
On the product of spheres, let $\Delta=\sum_{v=1}^n\Delta_v$ be the Laplace--Beltrami operator and
$Pf=f-\mathbb Ef$. Define $(-\Delta)^{-1}$ on the orthogonal complement of the constants, and set
\[
 D_g f=\sum_v\langle\nabla_v(-\Delta)^{-1}Pf,\nabla_v\theta_g\rangle.
\]
The spherical harmonic eigenvalues of the negative Laplacian are $\ell(\ell+d-2)$ \cite[Section 1.4]{DaiXu}.
The inverse above is therefore bounded on centered square-integrable functions.
All functions used here are finite linear combinations of product spherical harmonics, so differentiation and integration may be performed term by term.
Spherical integration by parts and the chain rule give, for centered $f$ and smooth $F$,
\begin{equation}\label{eq:stein-identity}
 \mathbb E[fF(\theta)]=\sum_{g\in E}\mathbb E[D_gf\,\partial_gF(\theta)].
\end{equation}

Write $\mathbb N_0=\{0,1,2,\ldots\}$. For a multi-index $\alpha\in\mathbb N_0^E$, let $|\alpha|=\sum_g\alpha_g$ and
$\alpha!=\prod_g\alpha_g!$. Starting from $b_{0,e,0}=\theta_e$, define recursively
\begin{equation}\label{eq:array-recursion}
 \begin{split}
 a_{j+1,e,\beta}&=\sum_{g:\,\beta_g>0}D_gb_{j,e,\beta-1_g},\\
 \mu_{j+1,e,\beta}&=\mathbb Ea_{j+1,e,\beta},\qquad
 b_{j+1,e,\beta}=a_{j+1,e,\beta}-\mu_{j+1,e,\beta},
 \qquad |\beta|=j+1.
 \end{split}
\end{equation}
Here $1_g$ is one in coordinate $g$ and zero elsewhere; there is no extra factor $\beta_g$ in the sum.
Define the array norms by
\[
 \|b_j\|^2=\sum_{e\in E}\sum_{|\alpha|=j}\alpha!\,\mathbb E b_{j,e,\alpha}^2,
 \qquad
 \|\mu_j\|^2=\sum_{e\in E}\sum_{|\alpha|=j}\alpha!\,\mu_{j,e,\alpha}^2.
\]
The factorial weight matches the Gaussian expansion below. Let $Z$ have $N$ independent standard Gaussian coordinates,
and let $H_\alpha$ be the multivariate Hermite polynomial in the probabilists' normalization. Then
$\mathbb E[H_\alpha(Z)H_\beta(Z)]=\alpha!\mathbf1_{\alpha=\beta}$.
For coefficients independent of $Z$, the above norm is therefore the $L^2$ norm of the corresponding Hermite expansion.
By \eqref{eq:stein-identity}, $\mu_{1,e,1_g}=h\mathbf1_{e=g}$.

A low-order example illustrates the coefficients. If $k(t)=at$, then
$(-\Delta)^{-1}\theta_{12}=\theta_{12}/[2(d-1)]$. Write
$t_{ij}=\langle X_i,X_j\rangle$. For $g=\{1,3\}$,
\[
 b_{1,\{1,2\},1_{\{1,3\}}}
 =D_{\{1,3\}}\theta_{12}
 =\frac{a^2}{2(d-1)}(t_{23}-t_{12}t_{13}).
\]
This depends only on three vertices and vanishes upon integration over either $X_2$ or $X_3$.
These vertices are precisely the leaves of the path $2\!-\!1\!-\!3$; we retain this property at every order.

Associate a multigraph of edge occurrences to $(e,\alpha)$: the edge $e$ occurs once, and the other occurrences are specified by $\alpha$.
Only connected labels can produce nonzero coefficients. Each coefficient depends only on vertices in its label
and has conditional mean zero in every vertex of degree one, counting multiplicity.
This follows by induction. Conditional averaging $\mathbb E_v$ commutes with $\Delta$ and its centered inverse,
since on the product spherical harmonic basis it retains only terms constant in coordinate $v$.
If the added edge does not touch an existing leaf, centering in that leaf is preserved.
If a new vertex $w$ is attached, then
$\mathbb E_w\nabla_v k(\langle X_v,X_w\rangle)=0$, so centering also holds at the new leaf.
If any leaf remains, the overall mean is already zero; if none remains, subtracting a constant preserves overall centering.
An edge disjoint from the current support has zero gradient pairing, so connectedness is also preserved.

\subsubsection{Adding new vertices and internal edges}
Let $F_j$ be the part of \eqref{eq:array-recursion} that adds an edge with one new endpoint,
and $I_j$ the part whose endpoints both lie in the existing support. Then
\begin{equation}\label{eq:FI}
 b_{j+1}=F_jb_j+PI_jb_j,\qquad \mu_{j+1}=\mathbb E I_jb_j.
\end{equation}
The new-vertex part is automatically centered. More precisely, the array space at order $j$ consists of
arrays of finite spherical harmonic expansions $u=(u_{e,\alpha})_{|\alpha|=j}$ satisfying the following conditions: every coordinate is centered
and depends only on vertices in its label; coordinates with disconnected labels are zero; and for every degree-one vertex $v$,
$\mathbb E_vu_{e,\alpha}=0$. The norm on this space is
$\|u\|^2=\sum_{e,\alpha}\alpha!\,\mathbb E u_{e,\alpha}^2$.
Both $F_j$ and $I_j$ act by the summation rule in \eqref{eq:array-recursion},
with the same factorial-weighted norm on their outputs. The output of $I_j$ may have nonzero mean before centering.

\begin{lemma}\label{lem:operators}
There is a constant $C$ depending only on $L,b$ such that, for every $j\ge0$,
\begin{equation}\label{eq:operator-bounds}
 \|F_j\|\le C\sqrt{nh/d},\qquad
 \|I_j\|\le C(j+1)h^{1/(2L)}.
\end{equation}
\end{lemma}
\begin{proof}
First fix a parent coordinate $f$ whose label has $l$ leaves, and let $U=(-\Delta)^{-1}f$.
Conditional centering at each leaf implies that every nonzero product harmonic term is nonconstant in all these $l$ coordinates.
If $l=0$, overall centering still guarantees at least one nonconstant coordinate. Hence
\begin{equation}\label{eq:leaf-gap}
 \mathbb E\sum_v|\nabla_vU|^2
 \le\frac{\mathbb Ef^2}{\max(1,l)(d-1)}.
\end{equation}
Each leaf coordinate contributes at least $d-1$ to the eigenvalue, so the spectral gap adds over the leaves.

Let
\[
 a=\mathbb E|\nabla_x k(\langle x,Y\rangle)|^2
 =\sum_{\ell=1}^L\ell(\ell+d-2)\frac{a_\ell^2}{M_{\ell,d}}
 \le L^2(d-1)h.
\]
Rotational invariance makes the gradient covariance of the new point $X_w$, on the tangent space at $X_v$,
equal to $a/(d-1)$ times the identity. Since $U$ does not depend on $X_w$,
\[
 \mathbb E|\langle\nabla_vU,\nabla_v\theta_{vw}\rangle|^2
 =\frac{a}{d-1}\mathbb E|\nabla_vU|^2.
\]
Summing over all old endpoints $v$ and new endpoints $w$, and using \eqref{eq:leaf-gap}, gives the bound
$na\mathbb Ef^2/[\max(1,l)(d-1)^2]$.

We must also control the sum of different parent coordinates contributing to the same child coordinate.
Each such parent is obtained by deleting a leaf edge other than the output edge; its multiplicity must be one.
If the child has $l'$ leaves, the number $r$ of deletable edges is at most $l'$.
For any of its parents, $l'\le l+1\le2\max(1,l)$.
Apply $|\sum_{i=1}^r u_i|^2\le r\sum_i|u_i|^2$
and reorder the sum by parents.
The new edge was absent from the parent, so the parent and child have the same factorial weight.
The leaf factor cancels the denominator in \eqref{eq:leaf-gap}, giving
\[
 \|F_jf\|^2\le\frac{2na}{(d-1)^2}\|f\|^2
 \le4L^2\frac{nh}{d}\|f\|^2.
\]
This explains the absence of any loss in the expansion order in the first bound.

For internal edges, let $D=\|k'\|_\infty$.
If $g=\{v,w\}$, then
$|D_gf|^2\le2D^2(|\nabla_vU|^2+|\nabla_wU|^2)$.
A parent graph has $j+1$ edge occurrences and at most $j+2$ vertices, so each vertex has at most $j+1$ possible internal neighbors.
Weighted Cauchy--Schwarz at a child coordinate gives
\[
 \left|\sum_{g:\beta_g>0}u_g\right|^2
 \le(j+1)\sum_{g:\beta_g>0}\frac{|u_g|^2}{\beta_g},
 \qquad \frac{\beta!}{\beta_g}=(\beta-1_g)!.
\]
To display the two counting factors, denote a parent coordinate by $u_{e,\alpha}$, its vertex set by $V_{e,\alpha}$,
and let $U_{e,\alpha}=(-\Delta)^{-1}u_{e,\alpha}$.
The factorial identity allows the child sum to be reordered by parents, yielding
\begin{align*}
 \|I_ju\|^2
 &\le (j+1)\sum_{e,\alpha}\alpha!
       \sum_{g\in\binom{V_{e,\alpha}}2}\mathbb E|D_gu_{e,\alpha}|^2\\
 &\le 2D^2(j+1)^2\sum_{e,\alpha}\alpha!
       \mathbb E\sum_{v\in V_{e,\alpha}}|\nabla_vU_{e,\alpha}|^2.
\end{align*}
In the second line, each vertex occurs at most $j+1$ times in the internal-edge sum.
Apply \eqref{eq:leaf-gap} to each centered parent coordinate to obtain
\begin{equation}\label{eq:internal-bound}
 \|I_ju\|^2\le\frac{2D^2(j+1)^2}{d-1}\|u\|^2.
\end{equation}
Fixed-degree polynomial interpolation gives $D\le C_L\|k\|_\infty$.
By $|P_{\ell,d}|\le1$ and Cauchy--Schwarz,
$\|k\|_\infty^2\le(\sum_{\ell=1}^LM_{\ell,d})h\le C_Ld^Lh$.
Combining this with $\|k\|_\infty\le b$ yields
\[
 \frac{D^2}{d-1}\le C_{L,b}\min\{d^{-1},h d^{L-1}\}
 \le C_{L,b}h^{1/L}.
\]
The last step considers $d\ge h^{-1/L}$ and $d<h^{-1/L}$ separately; when $h=0$, the kernel is zero.
Substitution in \eqref{eq:internal-bound} proves the second bound.
\end{proof}

\subsubsection{Mean coefficients and cumulants}
We now identify the constant terms in the recursion, allowing the preceding graph sums to be applied to these arrays.
For the bounded vector $\theta$, the moment generating function $M(t)=\mathbb E e^{t\cdot\theta}$ is analytic near the origin.
One integration by parts with the exponential test function gives
\[
 \partial_e\log M(t)
 =\sum_g t_g\mu_{1,e,1_g}
  +\sum_g t_g\frac{\mathbb E[b_{1,e,1_g}e^{t\cdot\theta}]}{M(t)}.
\]
For any centered coordinate $b_{j,e,\alpha}$, the same identity gives
\[
 \mathbb E[b_{j,e,\alpha}e^{t\cdot\theta}]
 =\sum_g t_g\mathbb E[D_gb_{j,e,\alpha}e^{t\cdot\theta}].
\]
Split each coefficient on the right into its mean and centered part, then collect terms with $\beta=\alpha+1_g$.
This is exactly \eqref{eq:array-recursion}, with each $g$ occurring once.
Iterating to any finite order $r$ gives
\[
 \partial_e\log M(t)
 =\sum_{j=1}^r\sum_{|\alpha|=j}\mu_{j,e,\alpha}t^\alpha
  +\sum_{|\alpha|=r}t^\alpha
       \frac{\mathbb E[b_{r,e,\alpha}e^{t\cdot\theta}]}{M(t)}.
\]
Since $\mathbb Eb_r=0$, the remainder is $O(\|t\|^{r+1})$.
Comparing the finite Taylor coefficients yields
\begin{equation}\label{eq:mu-cumulant}
 \mu_{j,e,\alpha}
 =\frac{\cum(\theta_e,\theta_g:\ g\text{ occurs }\alpha_g\text{ times})}{\alpha!}.
\end{equation}
At every finite order, this identity follows from uniqueness of Taylor coefficients.

For a graph $H$ of total multiplicity $m=j+1$, choosing the output edge $e$ gives
$1/\alpha!=m_e/\prod_fm_f!$. Summing over all output edges,
the contribution of this graph to $\|\mu_j\|^2$ is exactly
$mC_H^2/\prod_fm_f!$.
Thus Lemma~\ref{lem:core-sum} is precisely the fixed-cycle-rank bound for the mean coefficients.

\begin{proposition}\label{prop:array-small}
Suppose $A_3\to0$, $A_4\to0$, and $G\to0$, and take $R=\lceil3\log n\rceil$. For every fixed
$w\ge1$ and $s>0$,
\begin{equation}\label{eq:array-small}
 \sum_{j=2}^Rw^j\|\mu_j\|^2\longrightarrow0,
 \qquad s^{-R}\|b_R\|^2\longrightarrow0.
\end{equation}
\end{proposition}
\begin{proof}
Let $H_0=\|b_0\|=\sqrt{Nh}\le n$ and set
$f=C\sqrt G$, $i=CRh^{1/(2L)}$.
These bound $F_j$ and $I_j$, respectively, up to order $R$.
Group the arrays by the cycle rank $t$ of their labels, writing $b_{j,t}$.
Adding a new vertex preserves $t$, while an internal edge increases it by one. Induction using \eqref{eq:FI} gives
\[
 \|b_{j,t}\|\le H_0\binom jt f^{j-t}i^t,\qquad
 \|\mu_{j,t}\|\le H_0\binom{j-1}{t-1}f^{j-t}i^t.
\]
The second bound uses the fact that the last step must be an internal edge; taking means and centering are contractions in $L^2$ norm.
In particular, $\|b_R\|\le H_0(f+i)^R$.

Since $h^3\le C\tau$ and eventually $A_4\le1$, we have $h\le Cn^{-2/3}$,
so $f+i\to0$. For fixed $s>0$, eventually $(f+i)^2/s<e^{-1}$,
and hence $s^{-R}\|b_R\|^2\le n^2e^{-R}\to0$.
No prescribed rate of convergence of $G$ to zero is required.

Fix $t_0=3L+1$, and write $\mu_j^{\rm high}$ for the part with $t\ge t_0$.
Summing the binomial bounds first over $j-t$ and then over $t$ gives
\[
 \sum_{j=2}^Rw^{j/2}\|\mu_j^{\rm high}\|
 \le\frac{H_0(\sqrt w\,i)^{t_0}}
              {[1-\sqrt w(f+i)]^{t_0}}
 \longrightarrow0.
\]
More explicitly, with $x=\sqrt w f$ and $y=\sqrt w i$, extending the sum to infinity gives
$H_0y^{t_0}/[(1-x)^{t_0-1}(1-x-y)]$, which is at most the displayed bound.
Its numerator is bounded by
$C n^{1-t_0/(3L)}(\log n)^{t_0}=Cn^{-1/(3L)}(\log n)^{3L+1}\to0$.

For $t<t_0$, \eqref{eq:mu-cumulant} and Lemma~\ref{lem:core-sum} give
\[
 \sum_{j=2}^Rw^j\|\mu_j^{\rm low}\|^2
 \le C_{L,b,t_0,w}(A_3+A_4^{1/6})\longrightarrow0.
\]
Different values of $t$ correspond to different array coordinates, so their squared norms add.
The sum of squared norms of the high-cycle-rank part is at most the square of the sum of its norms.
Combining the two parts proves the first assertion.
\end{proof}

\subsection{From the higher-order expansion to the full graph distribution}\label{sec:entropy}
Small higher-order coefficients do not by themselves imply impossibility.
We first prove a finite-order relative entropy inequality and then use an edgewise random map to return to Bernoulli graphs.
This gives total variation convergence of the full observation distribution, rather than only convergence of finitely many moments or statistics.

\subsubsection{Finite-order Gaussian comparison}
Write $\KL(P\|Q)=\int\log(dP/dQ)\,dP$ for relative entropy and
$\TV(P,Q)=\sup_B|P(B)-Q(B)|$ for total variation.
The finite-order expansion below is related to higher-order Stein kernels \cite{Fathi}.
We first add independent Gaussian noise to obtain a smooth density.
Spherical integration by parts expresses the discrepancy between the conditional mean and the linear Gaussian predictor as a finite Hermite expansion.
This discrepancy determines the difference between the logarithmic gradients of the smoothed and Gaussian densities.
Integrating its squared norm along the noise variance yields a relative entropy bound.
The expansion retains nonzero mean coefficients at every order, allowing the higher moments of the original vector to differ from Gaussian moments.

\begin{lemma}\label{lem:entropy}
Let $Z\sim N(0,I_N)$ be independent of the spherical variables. For every integer $R\ge2$ and $s>0$,
\begin{equation}\label{eq:entropy-bound}
 \KL\bigl(\mathcal L(\theta+\sqrt s Z)\,|\,N(0,(h+s)I_N)\bigr)
 \le \sum_{j=2}^R\frac{\|\mu_j\|^2}{2(j+1)s^{j+1}}
       +\frac{\|b_R\|^2}{2(R+1)s^{R+1}}.
\end{equation}
\end{lemma}
\begin{proof}
All operations are first performed for fixed $N,R$.
Let $Y_u=\theta+\sqrt uZ$ and $m_u=\mathbb E[\theta\mid Y_u]$.
Let $H_\alpha$ be the multivariate Hermite polynomial in the probabilists' normalization, satisfying
$\mathbb E[H_\alpha(Z)H_\beta(Z)]=\alpha!\mathbf1_{\alpha=\beta}$.
For an array $c=(c_{e,\alpha})$, write
$(c:H_j)_e=\sum_{|\alpha|=j}c_{e,\alpha}H_\alpha$.

Set $B_u=\mathbb E[b_1:H_1(Z)\mid Y_u]$.
For a smooth compactly supported test function $\varphi$, apply \eqref{eq:stein-identity} and then Gaussian integration by parts to obtain
\[
 \mathbb E[\theta_e\varphi(Y_u)]
 =u^{-1/2}\mathbb E\left[\sum_g(h\mathbf1_{e=g}+b_{1,e,1_g})
                         Z_g\varphi(Y_u)\right].
\]
Thus $m_u=(h/u)(Y_u-m_u)+B_u/\sqrt u$, or equivalently,
\begin{equation}\label{eq:posterior}
 m_u-\frac{h}{u+h}Y_u=\frac{\sqrt u}{u+h}B_u.
\end{equation}

For any centered coefficient $c$ and $|\alpha|=j$, inserting spherical integration by parts between two Gaussian integrations by parts gives
\begin{align*}
 \mathbb E[cH_\alpha(Z)\varphi(Y_u)]
 &=u^{j/2}\mathbb E[c\,\partial^\alpha\varphi(Y_u)]\\
 &=u^{-1/2}\sum_g\mathbb E[D_gc\,H_{\alpha+1_g}(Z)\varphi(Y_u)].
\end{align*}
These are finitely many integrations. The coefficients are square-integrable and Gaussian polynomials have all moments, so every term is integrable.
Using this identity $R-1$ times in \eqref{eq:array-recursion} gives
\[
 B_u=\mathbb E\left[
       \sum_{j=2}^Ru^{-(j-1)/2}\mu_j:H_j(Z)
       +u^{-(R-1)/2}b_R:H_R(Z)\,\middle|\,Y_u\right].
\]
Conditional expectation is an $L^2$ contraction; unconditional Hermite orthogonality then yields
\begin{equation}\label{eq:B-norm}
 \mathbb E\|B_u\|^2\le
 \sum_{j=2}^Ru^{-(j-1)}\|\mu_j\|^2+u^{-(R-1)}\|b_R\|^2.
\end{equation}
The same-order cross term between $\mu_R$ and $b_R$ vanishes because $\mathbb Eb_R=0$.
This uses unconditional independence of $Z$ and the spherical variables, not independence conditional on $Y_u$.

We now turn \eqref{eq:B-norm} into a relative entropy bound.
Let $p_u$ be the density of $Y_u$, $q_u$ that of $N(0,(u+h)I_N)$, and
$D(u)=\KL(p_u\|q_u)$.
Since $\theta$ is bounded, on every interval $0<s\le u\le t<\infty$,
the density and its derivatives have integrable bounds of Gaussian decay times a polynomial, and
$|\log p_u(y)|\le C(1+\|y\|^2)$.
Differentiation of entropy and integration by parts are therefore valid, with vanishing boundary terms.
The heat equation $\partial_up_u=\Delta p_u/2$ and
$\mathbb E\|Y_u\|^2=N(u+h)$ give
\[
 D'(u)=-\frac12\int|\nabla\log p_u|^2p_u+\frac{N}{2(u+h)}.
\]
Using $\int y\cdot\nabla p_u(y)\,dy=-N$, the right-hand side equals
$-\frac12\mathbb E\|\nabla\log p_u(Y_u)+Y_u/(u+h)\|^2$.
Differentiating the Gaussian mixture density directly gives
$\nabla\log p_u(Y_u)=(m_u-Y_u)/u$, so
\begin{equation}\label{eq:entropy-derivative}
 -D'(u)=\frac{\mathbb E\|m_u-hY_u/(u+h)\|^2}{2u^2}
 =\frac{\mathbb E\|B_u\|^2}{2u(u+h)^2}.
\end{equation}

It remains to determine the endpoint at infinity. Convexity of relative entropy gives
$\KL(p_u\|N(0,uI_N))\le\mathbb E\|\theta\|^2/(2u)=Nh/(2u)$.
Among isotropic Gaussian reference measures, variance $u+h$, which matches the second moment, minimizes cross entropy.
Hence $0\le D(u)\le Nh/(2u)\to0$.
Integrate \eqref{eq:entropy-derivative} to finite $t$, then let $t\to\infty$.
Using $u+h\ge u$ and \eqref{eq:B-norm}, integrate the finite sum term by term via
$\int_s^\infty u^{-j-2}\,du=s^{-j-1}/(j+1)$ to obtain \eqref{eq:entropy-bound}.
\end{proof}

By Proposition~\ref{prop:array-small}, when $A_3,A_4,G\to0$,
the right-hand side of \eqref{eq:entropy-bound} tends to zero.
Indeed, take $w=\max\{1,s^{-1}\}$. The first part is at most
$(2s)^{-1}\sum_{j=2}^Rw^j\|\mu_j\|^2$, and the remainder also vanishes.
Thus, for every fixed $s>0$,
\begin{equation}\label{eq:gaussian-small}
 \KL\bigl(\mathcal L(\theta+\sqrt sZ)\,|\,N(0,(h+s)I_N)\bigr)\longrightarrow0.
\end{equation}

\subsubsection{Comparison of Bernoulli graphs}\label{sec:transfer}
Return to arbitrary edge density $p\in[p_0,1-p_0]$.
Let $W$ be the original connection function, and write
$\eta=W-p$, $\sigma=\sqrt{p(1-p)}$, $\kappa=\eta/\sigma$.
Then $|\eta|\le1-p_0<1$, so the preceding lemma applies directly to $k=\eta$.
Adding subscripts $\eta$ or $\kappa$ to the respective spectral quantities, we have
\begin{equation}\label{eq:rescale}
 h_\eta=\sigma^2h_\kappa,\qquad
 \delta_\eta=\sigma^3\delta_\kappa,\qquad
 \tau_\eta=\sigma^4\tau_\kappa.
\end{equation}
Thus $|\delta|^{2/3}$, $\sqrt\tau$, and $h/d$ are all multiplied by the same factor $\sigma^2$,
which has uniform positive lower and upper bounds.

First consider the transfer of detection. Independently for each observed edge $A_e$, retain it with probability $1/2$,
and otherwise replace it by an independent $\operatorname{Bernoulli}(1-p)$ variable.
The original model becomes a graph with conditional connection probability $(1+\eta_e)/2$,
and the Erd\H{o}s--R\'enyi model becomes independent fair edges.
Total variation cannot increase under a random map, so Proposition~\ref{prop:detect} yields detection from the original observation.

For impossibility, use a different random map.
Take $c=p_0/2$ and $s=2c^2/\pi$, and map a real number $y$ to a Bernoulli edge of probability
$p+c\operatorname{sign}(y)$, randomizing independently across coordinates.
A centered independent Gaussian vector is mapped exactly to $G(n,p)$.
Conditional on the latent points, the output probability for $\eta_e+\sqrt sZ_e$ is
\[
 q(\eta_e),\qquad
 q(t)=p+c\{2\Phi(t/\sqrt s)-1\},
\]
where $\Phi$ is the standard normal distribution function.
The choice of $s$ ensures $q'(0)=1$; in fact,
$q'(t)=e^{-t^2/(2s)}$, and hence
\begin{equation}\label{eq:channel-error}
 |q(t)-p-t|\le\frac{|t|^3}{6s},\qquad c\le q(t)\le1-c.
\end{equation}
For $a\in[0,1]$ and $q\in(0,1)$, the inequality $\log x\le x-1$ gives
\[
\KL(\operatorname{Bernoulli}(a)\|\operatorname{Bernoulli}(q))
\le(a-q)^2/[q(1-q)].
\]
Comparing the two independent-edge models conditional on all latent points, then averaging and applying Pinsker's inequality,
bounds the total variation between the original and mapped graphs by
\[
 C_{p_0}\left(\sum_e\mathbb E|\eta_e|^6\right)^{1/2}
 \le C_{L,p_0}n h_\eta^{3/2}
 \le C_{L,p_0}(n^4\tau_\eta^2)^{1/4}.
\]
The second step uses Lemma~\ref{lem:moments}, and the third uses $h_\eta^3\le C\tau_\eta$.
The estimate remains valid even when $W$ takes the value zero or one at some points, since the denominator comes from $q$.

If all original signals vanish, \eqref{eq:rescale} gives
$A_{3,\eta},A_{4,\eta},G_\eta\to0$.
By \eqref{eq:gaussian-small}, data processing, and the preceding comparison,
\begin{equation}\label{eq:tv-upper}
 \TV(P_{n,d,W},G(n,p))
 \le\sqrt{\frac12\KL\bigl(\mathcal L(\eta+\sqrt sZ)\|N(0,(h_\eta+s)I_N)\bigr)}
       +C_{L,p_0}A_{4,\eta}^{1/4}\longrightarrow0.
\end{equation}
The detection direction has already been proved. Since
$A_3=(n|\delta|^{2/3})^3$, $A_4=(n\sqrt\tau)^4$, and $G=nh/d$,
their simultaneous convergence to zero, or divergence of their maximum, is equivalent to $n/n_*\to0$ or $n/n_*\to\infty$, respectively.
This completes the proof of Theorem~\ref{thm:main}.

\section{Conclusions and further questions}\label{sec:conclusion}
The Mao--Wu--Xu spectral conjecture seeks to describe all geometric information through the cubic trace associated with triangles.
The quadratic cancellation and positive-spectrum quartic families show that general spherical kernel sequences also require four-cycle and geometric scales.
For dense polynomial kernels of fixed degree, Theorem~\ref{thm:main} combines the three into an explicit $n_*$,
and proves impossibility for $n\ll n_*$ and strong detection for $n\gg n_*$.
The central question that remains is how geometric information emerges in the critical regime $n\asymp n_*$, and which statistics can extract it.

\subsection{The critical curve for the positive-spectrum quartic model}
The pure degree-four harmonic kernel provides a concrete model for critical behavior,
with neither spectral sign cancellation nor competition between multiple harmonic levels.
Fix $b\in(0,1)$ and let
\[
 W_d(t)=\frac{1+bP_{4,d}(t)}2,\qquad
 M_d=\dim\mathcal H_{4,d}=\frac{d(d-1)(d+1)(d+6)}{24}.
\]
Parameterize sample size by the geometric signal $g$, taking $n_d(g)=\lfloor gdM_d/b^2\rfloor$.
Then $n\tr K^2/d\to g$, while the triangle and four-cycle signals tend to zero.
This family therefore isolates critical behavior arising from global geometry.

\begin{question}\label{q:critical}
For every fixed $g>0$, does the following limit exist? If so, determine it:
\[
 V_b(g)=\lim_{d\to\infty}
 \TV\bigl(P_{n_d(g),d,W_d},G(n_d(g),1/2)\bigr).
\]
In particular, which $g$ satisfy $V_b(g)=0$, and which satisfy $V_b(g)=1$?
Is there a nonempty interval on which $0<V_b(g)<1$?
\end{question}

The main theorem describes the two regimes $g\to0$ and $g\to\infty$.
Appendix~\ref{app:critical} further proves that, when $b=1/4$, every sufficiently large fixed $g$ already satisfies
$V_b(g)=1$.
We do not determine the behavior at small fixed $g$, nor prove that the critical curve exists or takes only the values zero and one.
Determining $V_b$ would therefore go beyond the order $d^5$:
it would distinguish a gradual accumulation of geometric information from a sudden onset of detectability.
The finite spherical partition argument compares only the number of possible positions with the total signal.
Further progress requires a finer understanding of overlaps between two latent configurations and their effect on the likelihood.

\subsection{How large a polynomial degree is needed to detect global geometry?}
Throughout $d^5\ll n\ll d^{16/3}$, the pure quartic model is strongly detectable,
yet Appendix~\ref{app:projection} shows that every fixed-degree standardized polynomial mean gap tends to zero.
A natural next step is to determine the necessary growth of polynomial degree before conjecturing the difficulty of general algorithms.

Again let $P=P_{n,d,W_d}$, $Q=G(n,1/2)$, and $\Lambda=dP/dQ$,
and let $\Pi_{\le D}$ be the orthogonal projection in $L^2(Q)$ onto adjacency polynomials of degree at most $D$.

\begin{question}\label{q:degree}
Throughout $d^5\ll n\ll d^{16/3}$, determine the range of $D=D(n,d)\to\infty$ for which
\[
 \|\Pi_{\le D}(\Lambda-1)\|_{L^2(Q)}\longrightarrow0.
\]
In particular, does this remain true for $D=c\log n$ with some fixed $c>0$?
At the scale where the norm no longer tends to zero, can one construct tests whose sum of errors tends to zero?
\end{question}

At fixed degree, the compressed cores have bounded size. As $D$ grows, both the number of core types and the moment-comparison constants
enter the estimates, so the finite-type argument in the appendix cannot be applied unchanged.
Controlling growing degrees would also connect this example with the low-degree likelihood ratio method \cite{KWB}.
Interpreting that method as a statement about general algorithmic complexity requires additional assumptions:
the question above is first a precise orthogonal-projection problem, and only then an algorithmic one.
The geometric test in the main text, by contrast, enumerates all region assignments.
It is also natural to ask whether a computationally tractable relaxation using spherical harmonic features can attain the same $d^5$ sample-size scale.

\subsection{Does average degree govern the sparse quartic model?}
Sparsity cannot be handled by simply replacing $p_0$ in the main theorem with a sequence tending to zero.
Work on threshold connections \cite{LMSY,DMSWX} shows that edge density changes the geometric estimates needed for impossibility.
For the polynomial kernels considered here, a concrete starting point is to retain the quartic harmonic structure while lowering the mean edge density.

Fix $b\in(0,1)$, let $p=p_n\to0$, and consider
\[
 W_{n,d}(t)=p\{1+bP_{4,d}(t)\}.
\]
For $p\le(1+b)^{-1}$, this is a valid $[0,1]$-valued connection function with mean exactly $p$.
Its standardized kernel is
$\kappa=b\sqrt{p/(1-p)}P_{4,d}$, so
\[
 \frac{d}{\tr K^2}
 =\frac{dM_{4,d}(1-p)}{b^2p}\asymp\frac{d^5}{p}.
\]
This suggests that the average degree $np$ may play the role of sample size $n$ in the dense setting, but does not establish a sparse theorem.

\begin{question}\label{q:sparse}
For this model, suppose $d\to\infty$, $p\to0$, and $np\to\infty$.
Is it true that
\[
 \TV\bigl(P_{n,d,W_{n,d}},G(n,p)\bigr)\longrightarrow
 \begin{cases}
 0,& np/d^5\longrightarrow0,\\[2pt]
 1,& np/d^5\longrightarrow\infty.
 \end{cases}
\]
If not, which additional sparse structure changes this scale?
\end{question}

The explicit spectrum and absence of cancellation help separate two difficulties:
weak high-dimensional geometric information, and rare structures among few edges that are not controlled by low-order spectral quantities.
Our relative entropy comparison back to Bernoulli graphs uses a fixed positive lower bound on edge density.
Answering Question~\ref{q:sparse} requires a comparison valid as $p\to0$, or a direct analysis of the sparse likelihood.

\subsection{The three-scale principle beyond fixed degree}
Finally, the role of fixed degree should be distinguished from the spectral formulation of the conjecture.
Formula~\eqref{eq:sample-scale} remains meaningful for many nonpolynomial kernels,
but the uniform moment comparison and higher-order remainder bounds currently depend on fixed $L$.
Uniformly analytic connection functions form a controlled class in which to seek an extension.

\begin{question}\label{q:analytic}
Fix $r>1$, $B>0$, and $p_0\in(0,1/2]$.
Let $W_n(t)=\sum_{\ell\ge0}c_{\ell,n}t^\ell$ satisfy
\[
 \sum_{\ell\ge0}|c_{\ell,n}|r^\ell\le B,\qquad
 0\le W_n\le1,\qquad p_0\le\mathbb EW_n(\langle X_1,X_2\rangle)\le1-p_0.
\]
Define $n_*$ by the same formula \eqref{eq:sample-scale}.
Do the two separation conclusions of Theorem~\ref{thm:main} still hold?
\end{question}

Uniform analyticity allows polynomials of degree of order $\log n$ to attain the approximation accuracy needed to compare graph distributions.
This does not automatically extend the main theorem: one must also prove stability of the spectral scales under approximation
and control the constants as the degree grows.
An answer would clarify whether the three-scale formula reflects a special feature of fixed degree
or a broader principle of spherical detection.

\clearpage
\appendix
\section*{Appendices}
The main text establishes three detection scales for spherical polynomial kernels of fixed degree.
Two related questions remain: why can global geometry contain information beyond that of fixed-degree statistics,
and can strong detection still occur when the geometric signal neither vanishes nor diverges?
The two appendices address these questions in turn.

Starting from the orthogonal expansion of the likelihood ratio, Appendix~\ref{app:projection} proves that simultaneous vanishing of the triangle and four-cycle signals is equivalent to vanishing of every fixed-degree nonconstant likelihood ratio projection.
Together with the quartic example, this shows that all fixed-degree standardized polynomial mean gaps can tend to zero while the total variation distance of the full graph distributions tends to one.
Appendix~\ref{app:critical} further studies the same quartic family and proves strong detection when the geometric signal tends to a sufficiently large finite constant.
This shows that the two extreme regimes of the main theorem do not exhaust the behavior at the critical scale, and provides a concrete starting point for the critical-curve question in Section~\ref{sec:conclusion}.

\section{Separation between fixed-degree statistics and the full observation}\label{app:projection}
The quartic example raises a natural question: after triangles and four-cycles fail,
can one use a more complicated polynomial statistic whose degree is still fixed?
We show that the answer is negative in the sense of standardized mean gaps.
Higher fixed-degree likelihood ratio projections do not reveal the third scale, although the full graph distribution remains detectable.

Let $Q=G(n,p)$, $P=P_{n,d,W}$, and $\Lambda=dP/dQ$.
Since $Q$ is positive at every point of the finite graph space, the likelihood ratio always exists.
Write $Z_e=(A_e-p)/\sqrt{p(1-p)}$.
For a simple edge set $F\subseteq\binom{[n]}2$, let $Z_F=\prod_{e\in F}Z_e$, with $Z_\varnothing=1$.
These functions form an orthonormal basis of $L^2(Q)$: the single-edge space is spanned by $1,Z_e$,
and the independent-edge space is the tensor product of these two-dimensional spaces.

Let $\Pi_{\le D}$ be the orthogonal projection onto polynomials in the adjacency variables of degree at most $D$.
Here $D$ is the degree of the observed statistic, distinct from the degree bound $L$ on the connection kernel.
This is the orthogonal projection used in the low-degree likelihood ratio method \cite{KWB}.
Since $A_e^r=A_e$ for $r\ge1$, this space is spanned by $Z_F$ with $|F|\le D$.
Taking conditional expectations over the edges given the positions yields
\begin{equation}\label{eq:projection-identity}
 \|\Pi_{\le D}(\Lambda-1)\|_{L^2(Q)}^2
 =\sum_{1\le|F|\le D}t(F,\kappa)^2,\qquad
 t(F,\kappa)=\mathbb E\prod_{ij\in F}\kappa(\langle X_i,X_j\rangle).
\end{equation}
If $F$ has a leaf, its term vanishes.
The only nonempty simple graphs with at most four edges, no isolated vertices, and no leaves are the triangle and four-cycle. Thus
\begin{equation}\label{eq:projection-four}
 \|\Pi_{\le4}(\Lambda-1)\|_2^2
 =\binom n3\delta^2+3\binom n4\tau^2.
\end{equation}
These coefficients of random geometric graphs are also commonly called Fourier coefficients;
see \cite{BBfourier,BBalgebra} for estimates based on their graph structure.

\begin{proposition}\label{prop:fixed-degree}
Under the assumptions of Theorem~\ref{thm:main}, the following are equivalent:
\begin{enumerate}
 \item $n^3\delta^2+n^4\tau^2\to0$; 
 \item For every fixed integer $D\ge1$, $\|\Pi_{\le D}(\Lambda-1)\|_{L^2(Q)}\to0$.
\end{enumerate}
\end{proposition}
\begin{proof}
Taking $D=4$ in the second assertion and using \eqref{eq:projection-four} proves the first.
Conversely, fix $D$. There are only finitely many isomorphism types of simple graphs with at most $D$ edges and no isolated vertices,
and a type with $v$ vertices has at most $n^v$ labelings.
It suffices to prove $n^v t(F,\kappa)^2\to0$ for every such nonempty graph $F$.
Here $h,\delta,\tau$ refer to the standardized kernel $\kappa$;
its supremum is uniformly controlled by $p_0$, so Lemma~\ref{lem:moments} still applies.

Graph moments multiply over connected components. They also multiply over the blocks obtained by splitting a connected graph at its cut vertices:
conditional on a cut vertex, the sides are independent, and rotational invariance makes each conditional graph moment constant.
A bridge is a single-edge block with zero moment.
Otherwise, each block has at least three vertices and no cut vertex, and
the number of vertices in the original graph is at most the sum of the block vertex counts.
It therefore suffices to treat one such block.

If the block is a cycle of length $m$, its moment is $\tr K^m$.
For $m=3$, $n^m t^2=A_3\to0$; for $m\ge4$,
$n^m t^2\le A_4\rho^{m-4}\to0$, where $\rho=n\|K\|_{\mathrm{op}}^2\le A_4^{1/4}$.
If the block is not a cycle, suppress its degree-two paths and use the notation of Lemma~\ref{lem:core-sum}.
Here only a single graph moment is estimated, with no sum over partitions.
Path convolution and H\"older's inequality again give
\[
 t(F,\kappa)^2\le C_{L,D,p_0}h^{R_0}\tau^p\|K\|_{\mathrm{op}}^{2z}.
\]
Each path has two distinct endpoints, for the same reason as in the main proof; core vertices have degree at least three, so
$v_0\le2(R_0+p)/3$.
Using $h^3\le C\tau$, we obtain
\[
 n^v t(F,\kappa)^2
 \le C n^{v_0-2R_0/3-p}A_4^{R_0/6+p/2}\rho^z\longrightarrow0.
\]
The exponent of $n$ is nonpositive and that of $A_4$ is strictly positive.
Multiplying the estimates over the finitely many blocks completes the proof.
\end{proof}

The norm has a direct statistical interpretation. Orthogonal projection and Cauchy--Schwarz give
\[
 \sup_{\substack{\deg T\le D,\ \mathbb E_QT=0\\\mathbb E_QT^2\le1}}
     |\mathbb E_PT-\mathbb E_QT|
 =\|\Pi_{\le D}(\Lambda-1)\|_2.
\]
For the positive-spectrum quartic sequence in \eqref{eq:quartic-gap}, every fixed-degree standardized mean gap therefore tends to zero,
while $\TV(P,Q)\to1$.
This statement does not rule out arbitrary nonlinear processing of polynomial outputs and is not a computational complexity lower bound.
It describes precisely the limitation of polynomial mean tests normalized in $L^2(Q)$.
Related separations between statistical detection and low-degree methods occur in other algebraic graph models \cite{BBalgebra};
the present example realizes such a separation on the sphere with a fixed quartic kernel and a spectrum of one sign.

\section{Strong detection at a finite geometric scale}\label{app:critical}
Theorem~\ref{thm:main} describes the two extremes in which all three signals vanish or at least one diverges.
The region between them cannot simply be treated as a universal finite-signal regime.
Even when the triangle and four-cycle signals vanish, if the geometric signal tends to a sufficiently large finite constant,
total variation can still tend to one.

\begin{proposition}\label{prop:finite-G}
There is a fixed constant $g_0>0$ such that for every fixed $g>g_0$ there is a sequence of positive-spectrum quartic kernels satisfying
\[
 n^3\delta^2\to0,\qquad n^4\tau^2\to0,\qquad nh/d\to g,
 \qquad \TV(P_k,G(n,1/2))\to1.
\]
\end{proposition}
\begin{proof}
Take $k=bP_{4,d}$ with $b=1/4$, as in the main text, let $d\to\infty$, and choose
$n=\lfloor gdM_{4,d}/b^2\rfloor$.
By \eqref{eq:quartic-traces},
\[
 \frac{nh}{d}\to g,\qquad
 n^3\delta^2\sim\frac{g^3d^3}{M_{4,d}}\to0,\qquad
 n^4\tau^2\sim\frac{g^4d^4}{M_{4,d}^2}\to0.
\]
Lemma~\ref{lem:partition} gives $\log J\le c_4d$.
For all sufficiently large $n$, $Nh/16\ge n^2h/64$ and $nh\ge gd/2$, so
\[
 n\log J-Nh/16\le nd(c_4-g/128).
\]
Taking $g_0=128c_4$ makes this exponent tend to minus infinity; also $nh\to\infty$.
Lemma~\ref{lem:scan} then gives strong detection.
\end{proof}

In particular, $n/n_*$ tends to the finite constant $g$, rather than to infinity.
The strong-detection condition in the main theorem is therefore not necessary.
A finer description of the quartic model at $n\asymp d^5$ requires comparing the number of latent configurations with the information in the observation,
rather than merely substituting three spectral quantities into a criterion with a fixed numerical cutoff.

\clearpage
\providecommand{\bysame}{\leavevmode\hbox to3em{\hrulefill}\thinspace}
\providecommand{\MR}{\relax\ifhmode\unskip\space\fi MR }
\providecommand{\MRhref}[2]{  \href{http://www.ams.org/mathscinet-getitem?mr=#1}{#2}
}
\providecommand{\href}[2]{#2}


\begin{thebibliography}{10}

\bibitem{BBcube}
K.~Bangachev and G.~Bresler, \emph{{Detection of $L_\infty$ geometry in random
  geometric graphs: suboptimality of triangles and cluster expansion}},
  Proceedings of Thirty Seventh Conference on Learning Theory, Proceedings of
  Machine Learning Research, vol. 247, PMLR, 2024, pp.~427--497.

\bibitem{BBfourier}
K.~Bangachev and G.~Bresler, \emph{{On the Fourier coefficients of
  high-dimensional random geometric graphs}}, Proceedings of the 56th Annual
  ACM Symposium on Theory of Computing, ACM, 2024, pp.~549--560.

\bibitem{BBalgebra}
K.~Bangachev and G.~Bresler, \emph{{Random algebraic graphs and their
  convergence to Erd{\H{o}}s--R{\'e}nyi}}, Random Structures \& Algorithms
  \textbf{66} (2025), no.~1, article no.~e21276.

\bibitem{BDER}
S.~Bubeck, J.~Ding, R.~Eldan, and M.~Z. R{\'a}cz, \emph{{Testing for
  high-dimensional geometry in random graphs}}, Random Structures \& Algorithms
  \textbf{49} (2016), no.~3, 503--532.

\bibitem{DaiXu}
F.~Dai and Y.~Xu, \emph{{Approximation theory and harmonic analysis on spheres
  and balls}}, Springer Monographs in Mathematics, Springer, New York, 2013.

\bibitem{DMSWX}
H.~Du, C.~Mao, N.~Sun, Y.~Wu, and J.~Xu, \emph{{Resolution of the detection
  threshold conjecture for random geometric graphs in the $d>n$ regime}},
  preprint, 2026, arXiv:2607.02013v1.

\bibitem{Fathi}
M.~Fathi, \emph{{Higher-order Stein kernels for Gaussian approximation}},
  Studia Mathematica \textbf{256} (2021), no.~3, 241--258.

\bibitem{KWB}
D.~Kunisky, A.~S. Wein, and A.~S. Bandeira, \emph{{Notes on computational
  hardness of hypothesis testing: predictions using the low-degree likelihood
  ratio}}, Mathematical Analysis, its Applications and Computation
  (P.~Cerejeiras and M.~Reissig, eds.), Springer Proceedings in Mathematics \&
  Statistics, vol. 385, Springer, Cham, 2022, pp.~1--50.

\bibitem{LMSY}
S.~Liu, S.~Mohanty, T.~Schramm, and E.~Yang, \emph{{Testing thresholds for
  high-dimensional sparse random geometric graphs}}, Proceedings of the 54th
  Annual ACM SIGACT Symposium on Theory of Computing, ACM, 2022, pp.~672--677.

\bibitem{LRlatent}
S.~Liu and M.~Z. R{\'a}cz, \emph{{A probabilistic view of latent space graphs
  and phase transitions}}, Bernoulli \textbf{29} (2023), no.~3, 2417--2441.

\bibitem{MWX}
C.~Mao, Y.~Wu, and J.~Xu, \emph{{Random geometric graphs with smooth kernels:
  sharp detection threshold and a spectral conjecture}}, preprint, 2026,
  arXiv:2602.14998v1.

\end{thebibliography}
\end{document}